\documentclass[11pt]{article}
\usepackage{listings}
\usepackage{pdflscape}
\usepackage{amsfonts}
\usepackage{epsfig}
\usepackage{lscape}
\usepackage{longtable}
\usepackage{amsmath,amssymb,amsthm}
\usepackage{graphicx}
\usepackage{caption}
\usepackage[labelformat=simple]{subcaption}
\usepackage{color}
\usepackage{placeins}
\usepackage{url}
\usepackage{cases}
\usepackage{longtable}
\usepackage{supertabular}
\usepackage{multirow}

\usepackage{amsmath}
\allowdisplaybreaks[4]
\usepackage{cite}
\usepackage{algorithmic}
\usepackage{algorithm,float}
\usepackage{booktabs}
\usepackage{enumitem}
\usepackage{verbatim}



\newtheorem{theorem}{Theorem}
\newtheorem{assumption}{Assumption}

\newtheorem{remark}{Remark}

\def\norm#1{\|#1\|}
\def\myfbox#1{\fbox{\begin{minipage}{16cm}#1\end{minipage}}}
\providecommand{\keywords}[1]
{
  \textbf{Keywords:} #1
}

\providecommand{\AMSCode}[1]
{
  \textbf{AMS subject classification:} #1
}

\begin{document}

\title{\bf A Dimension Reduction Technique for Large-scale Structured Sparse Optimization Problems with Application to Convex Clustering
}%\footnote{A preliminary version which contains partial results is under review at NeurIPS 2021. Please see the cover letter and the attached conference version for details.}\footnotemark[1]
\author{ Yancheng Yuan\footnotemark[2], \quad
Tsung-Hui Chang\footnotemark[3], \quad
Defeng Sun\footnotemark[4], \quad Kim-Chuan Toh\footnotemark[5]}
\date{\today}
\maketitle

\renewcommand{\thefootnote}{\fnsymbol{footnote}}
%\footnotetext[1]{{\bf Funding:} Yancheng Yuan is supported by the National Research Foundation, Singapore under its International Research Centres in Singapore Funding Initiative, Defeng Sun is supported in part by the Hong Kong Research Grant Council grant PolyU 153014/18P and Kim-Chuan Toh by ARF grant R146-000-257-112 of the Ministry of Education of Singapore.}
\footnotetext[2]{Department of Applied Mathematics, The Hong Kong Polytechnic University, Hung Hom, Hong Kong ({\tt yancheng.yuan@polyu.edu.hk}).}
\footnotetext[3]{School of Science and Engineering, The Chinese University of Hong Kong (Shenzhen) and Shenzhen Research Institute of Big Data, China \texttt{changtsunghui@cuhk.edu.cn}. The research of this author is in part supported by the Shenzhen Research Institute of Big Data, under Grant 2019ORF01002.}
\footnotetext[4]{Department of Applied Mathematics, The Hong Kong Polytechnic University, Hung Hom, Hong Kong ({\tt defeng.sun@polyu.edu.hk}).  The research of this author is supported in part by Hong Kong Research Grant Council under grant number 15303720 and the Shenzhen Research Institute of Big Data, China under Grant 2019ORF01002.}
\footnotetext[5]{Department of Mathematics and Institute of Operations Research and Analytics, National University of Singapore, 10 Lower Kent Ridge Road, Singapore ({\tt mattohkc@nus.edu.sg}).The research of this author is supported by the Ministry of Education, Singapore, under its Academic Research Fund Tier 3 grant call (MOE-2019-T3-1-010).}
\renewcommand{\thefootnote}{\arabic{footnote}}

\begin{abstract}
 In this paper, we propose a novel adaptive sieving (AS) technique and an enhanced AS (EAS) technique, which are solver independent and could accelerate optimization algorithms for solving large scale convex optimization problems with intrinsic structured sparsity. We establish the finite convergence property of the AS technique and the EAS technique with inexact solutions of the reduced subproblems. As an important application, we apply the AS technique and the EAS technique on the convex clustering model, which could accelerate the state-of-the-art algorithm {\sc Ssnal} by more than 7 times and the algorithm ADMM by more than 14 times.
\end{abstract}
\keywords{Adaptive sieving, structured sparsity, dimension reduction, convex optimization, convex clustering.}\\[3pt]
\AMSCode{90C06, 90C25, 90C90}
\section{Introduction and Related Work}
Clustering is one of the most important and fundamental problems in data science, which plays important roles in numerous applications. Significant advances have been achieved in clustering during the last few decades, including K-means \cite{lloyd1982least,vassilvitskii2006k}, spectral clustering \cite{ng2002spectral,shi2000normalized}, subspace clustering \cite{soltanolkotabi2014robust,vidal2011subspace} and so on. Despite these developments, some known drawbacks of these centroid based models, such as sensitivity to the initialization, limited effectiveness in high dimensional problems, and more importantly, the requirement on prior knowledge of the number of clusters, are still challenging to overcome. Here, we want to emphasize that the requirement on prior knowledge of the number of clusters is impractical for most real applications and to estimate the number of clusters itself is as important as clustering. One could argue that we can run classical clustering algorithms, such as K-means, with a few guesses on the number of clusters, but these clustering results are usually independent. Thus, users still need to determine the final clustering results subjectively based on their own preference.

Recently, the convex clustering approach has been proposed \cite{hocking2011clusterpath,lindsten2011clustering,pelckmans2005convex} and becomes more and more popular due to its good empirical performance and nice theoretical guarantees \cite{chi2019recovering,jiang2020recovery,panahi2017clustering,sun2021convex,tan2015statistical,zhu2014convex}. More recently, some nonconvex variants of the convex clustering model have also been proposed \cite{lin2020centroid,shah2017robust}. Specifically, for a given collection of $n$ data points which are put as columns of a matrix $A \in \mathbb{R}^{d \times N}$, the convex clustering model is to solve the following optimization problem
\begin{equation}
\label{eq: convex-clustering}
\min_{X \in \mathbb{R}^{d \times N}} \quad \frac{1}{2}\sum_{i=1}^N \|X_{:i} - A_{:i}\|^2 + \lambda\sum_{1 \leq i < j \leq N}w_{ij}\|X_{:i} - X_{:j}\|_p,
\end{equation}
where $X_{:i}$ (or $A_{:i}$) is the $i$-th column of $X$ (or A),  $w_{ij} = w_{ji} \geq 0$ are given weights and  $\lambda \geq 0$ is the hyper-parameter to control the strength of the diffusion penalty. Here $\|\cdot\|_p$ is the vector $p$-norm and we require $p \geq 1$ to guarantee the convexity of the model. After solving the model \eqref{eq: convex-clustering} and obtaining the solution $X^*$, we assign the data points $A_{:i}$ and $A_{:j}$ to the same cluster if $X_{:i}^* = X_{:j}^*$. Readers who are interested in more details about cluster identification based on the convex clustering model with an inexact solution could refer to \cite{chi2015splitting,jiang2020sum,sun2021convex}. It has been proved in \cite{chi2015splitting} that the convex clustering model \eqref{eq: convex-clustering} can generate a continuous clustering path with respect to the hyper-parameter $\lambda$. Thus, prior knowledge on the number of clusters is \textbf{not} required, which is a highly desirable property.

Although the convex clustering model \eqref{eq: convex-clustering} is strongly convex, it is still quite challenging to solve since the number of terms in the diffusion penalty grows with $n$ and could be extremely large (up to $O(n^2)$). A number of numerical optimization algorithms has been proposed for solving the convex clustering model. Among them, the alternating direction method of multipliers (ADMM) and the alternating minimization algorithm (AMA), which are proposed in \cite{chi2015splitting}, are very popular. Recently, a second-order based algorithm called {\sc Ssnal} has been proposed \cite{yuan2018efficient, sun2021convex}, which can efficiently solve \eqref{eq: convex-clustering} to achieve high accuracy for large $n$ but
moderate $d$, by taking advantage of the so-called second-order sparsity. However, the scalability could still be limited for those algorithms due to their need to handle all data
points simultaneously.
On the other hand, one may try to use some stochastic algorithms to solve \eqref{eq: convex-clustering} as in \cite{panahi2017clustering}, however, the empirical performance is not so attractive since we need a rather accurate solution in order to determine the cluster memberships correctly based on the obtained solution $X^*$. Naturally one would ask whether we could design a deterministic algorithm which can scale as well as those stochastic algorithms.
While this goal seems unattainable at the first glance, here we will give a positive answer. Now, we briefly explain the key idea on why this is possible. As demonstrated later in this paper, the same idea works not only for convex clustering model, but also for other optimization problems with special structures.

The key idea behind could actually be explained in a single sentence, that is, although the number of data points $N$ could be extremely large, the number of clusters must be small, which is the purpose of clustering. In other words, for well-chosen values of $\lambda$, most columns of the optimal solution $X^*$ for model \eqref{eq: convex-clustering} should be identical. In this case, most of the terms in the diffusion penalty should be zero. If we can remove those zero terms in advance and reduce the dimension of $X$ simultaneously, we only need to solve a small scale optimization problem, even for extremely large $N$. In this paper, we will propose an adaptive sieving (AS) technique and an enhanced adaptive sieving (EAS) technique, which are rigorous implementations of the aforementioned idea with theoretical guarantees. The details could be found later in this paper. We would like to emphasize that, the dimension reduction techniques proposed in this paper are solver independent, thus they could be applied to various algorithms which can solve the reduced optimization problems.

Motivated by the convex clustering problem, in this paper, we consider the optimization problems of the following form:
\begin{equation}\label{eq: two-block-structured}
\min_{x \in \mathbb{R}^n} \quad  F_{\lambda}(x) := f(x) + \lambda p(Bx),
\end{equation}
where $\lambda > 0$ is a hyper-parameter, $f : \mathbb{R}^n \to \mathbb{R}$ is a twice continuously differentiable convex function, $p : \mathbb{R}^m \to (-\infty, +\infty]$ is a closed proper convex function and $B: \mathbb{R}^n \to \mathbb{R}^m$ is a linear map. In some real applications, $p$ is usually a regularizer which can enforce sparsity on $Bx$ and $B$ is a linear map which encodes desirable structures of $x$. This indicates the meaning of structured sparsity. With special designed matrices $B$, the optimization problem \eqref{eq: two-block-structured} includes many important models, such as the convex clustering model \eqref{eq: convex-clustering} \footnote{We can take $x = vec(X) \in \mathbb{R}^{dN}$, where $vec(X)$ is the vectorization of the matrix $X$ by stacking its columns.}, fused lasso model \cite{tibshirani2005sparsity}, clustered lasso model \cite{clusteredlasso}, and so on.

In this paper, we will propose solver independent techniques which can solve the optimization problem \eqref{eq: two-block-structured} via solving a sequence of subproblems with much smaller problem size. The main idea of this paper is inspired by a recent preprint \cite{lin2020adaptive}. Authors in \cite{lin2020adaptive} introduce an adaptive sieving technique to reduce the dimension of the optimization problem with sparse solutions (taking the linear map $B$ to be identity $I$ in \eqref{eq: two-block-structured}). However, the same idea cannot be directly applied to \eqref{eq: two-block-structured} for the cases where $B \neq I$. First, the optimal solutions of \eqref{eq: two-block-structured} may not be sparse at all, only with some special structures, such as being block-wise constant. Thus, the adaptive sieving technique in \cite{lin2020adaptive} could not apply on $x$.  Second, one may try to apply the AS directly on $Bx$ by introducing a new variable $y = Bx$. Although this idea may work for reducing the dimension of $y$ (or $Bx$), this direct application cannot reduce the dimension of $x$ simultaneously. If we cannot reduce the dimension of $x$ simultaneously, we still need to solve large scale subproblems. Third, one of the keys for applying the AS technique is to check the optimality condition of \eqref{eq: two-block-structured} for a given $\bar{x} \in \mathbb{R}^n$. However, as one may see later, this is highly non-trivial if the inverse of $B$ is not available (which is the case for most of problems with structured sparsity). In this paper, we will propose a new adaptive sieving technique and an enhanced adaptive sieving technique to address all of these issues.

To demonstrate the effectiveness of the proposed idea, we evaluate the empirical performance of our proposed AS and EAS technique with state-of-the-art algorithms: {\sc Ssnal} \cite{yuan2018efficient}, ADMM \cite{chi2015splitting} and AMA \cite{chi2015splitting}, for solving the convex clustering model \eqref{eq: convex-clustering}. As the readers will see later in the numerical experiments section, the numerical results on both simulated and real data sets demonstrate that the proposed AS and EAS could substantially reduce the dimension of the optimization problems. As a result, the AS/EAS techniques can accelerate the state-of-the-art algorithm {\sc Ssnal} by more than 7 times and the algorithm ADMM by more than 14 times for solving the convex clustering model.

The main contributions of our paper can be summarized as follows:

\begin{itemize}
    \item We propose a new solver independent adaptive sieving (AS) technique which can be applied to solve large scale optimization problems \eqref{eq: two-block-structured} with structured sparsity by solving a sequence of subproblems with much smaller size.
    \item We show the details of how to reduce the dimension of $x$ and $Bx$ simultaneously. Also, we show how to construct the corresponding reduced subproblem of \eqref{eq: two-block-structured} based on the structured sparsity of $x$ (i.e., the sparsity of $Bx$).
    \item Our proposed AS technique allows the reduced subproblems to be solved inexactly and we prove the finite convergence property of the proposed AS technique for solving \eqref{eq: two-block-structured}.
    \item As one will see later, although the AS technique will converge in finite iterations for solving \eqref{eq: two-block-structured}, the sieving procedure of the AS technique may continue even if we obtain an optimal solution $x^*$ of \eqref{eq: two-block-structured}. To address this issue, we propose an enhanced adaptive sieving (EAS) technique, which can certify the optimality of an obtained solution with low additional computational cost. This can potentially reduce the sieving iterations of the AS technique and further accelerate the algorithms. The finite convergence property of the EAS technique is also proved.
    \item Both the AS technique and the EAS technique are extended to obtain a solution path of the structured sparse optimization problem \eqref{eq: two-block-structured} for a sequence of hyper-parameters $+\infty > \lambda_1 > \lambda_2 > \cdots > \lambda_k > 0$.
    \item As an important application, extensive numerical experiments on the convex clustering model for both simulated and real data sets are provided. The superior numerical experiment results demonstrate the power of the AS technique and the EAS technique for accelerating numerical optimization algorithms to generate the solution path for the convex clustering model \eqref{eq: convex-clustering}.
   % In particular, the average problem size of the reduced subproblems generated by the AS/EAS techniques are only about $1/10$ of the original problems. As a result, the AS/EAS techniques can accelerate the state-of-the-art algorithm {\sc Ssnal} by more than 7 times and the algorithm ADMM by more than 14 times.
\end{itemize}

The rest of this paper is organized as follows: In Section 2, we introduce the adaptive sieving technique and the enhanced adaptive sieving technique  for optimization problems with structured sparsity. The application of the AS technique and the EAS technique on the convex clustering model will be shown in Section 3 and numerical results are summarized in Section 4. We conclude the paper in Section 5.
\medskip

\textbf{Notation}. We use blackboard bold capital letters to denote finite dimensional real Euclidean spaces, e.g. $\mathbb{X}$, $\mathbb{Y}$. In particular, we use $\mathbb{R}^{m \times n}$ ($\mathbb{R}$) to denote the set of all real $m \times n$ matrices (real numbers). We denote column vectors by lowercase letters, e.g. $v \in \mathbb{R}^n$, and matrices by capital letters, e.g. $A \in \mathbb{R}^{m \times n}$. We denote the transpose of the matrix $A$ as $A^T$;
the $i$-th ($ij$-th) element of a vector $v$ (matrix $A$) by $v_i$ ($A_{ij}$). For a given integer $n \geq 1$, we denote the collection of integers from $1$ to $n$ by $[n]$. We denote the complement of an index set $I \subseteq [m]$ as $I^c$. For given index sets $I \subseteq [m]$ and $J \subseteq [n]$, we denote the submatrix consisting with rows (columns) indexed by $I$ ($J$) as $A_{I:}$ ($A_{:J}$). We denote the range space and null space of $A$ by ${\rm Range}(A)$ and ${\rm Null}(A)$, respectively. For a vector $x \in \mathbb{R}^n$ and a scalar $p > 0$, we define the vector $p$-norm as: $\|x\|_p := (\sum_{i=1}^n |x_i|^p)^{1/p}$. We use $\|\cdot\|$ to denote the vector $2$-norm. For a closed proper convex function $f: \mathbb{R}^n \to (-\infty, + \infty]$, the conjugate of $f$ is $f^*(z) := \sup_{x \in \mathbb{R}^n} \{\langle x, z \rangle - f(x)\}$. For a closed convex set $C \subseteq \mathbb{R}^n$ and a given vector $a \in \mathbb{R}^n$, the projection of $a$ onto the set $C$ is $\Pi_{C}(a) := \arg\min_{x \in C} \frac{1}{2}\|x - a\|^2$.

\section{An Adaptive Sieving Technique for Structured Sparsity}
\label{sec: AS-General}
In this section, we will introduce a novel adaptive sieving technique for obtaining the solution path for the structured sparse convex programming problem \eqref{eq: two-block-structured}.
 Equivalently, we can reformulate \eqref{eq: two-block-structured} as follows:
\begin{equation}\label{eq: two-block-structured-reform}
\tag{$P_{\lambda}$}
\begin{array}{ll}
\min_{x \in \mathbb{R}^n, y \in \mathbb{R}^m} & f(x) + \lambda p(y) \\
\rm{s.t.} & Bx - y = 0.
\end{array}
\end{equation}
The Lagrangian function corresponding to \eqref{eq: two-block-structured-reform} is defined as:
\begin{equation}\label{eq: LF-two-block-structured-reform}
l(x, y; z) := f(x) + \lambda p(y) + \langle z, Bx - y \rangle,
\end{equation}
where $z \in \mathbb{R}^m$ is the Lagrange multiplier. The corresponding dual problem is given by
\begin{equation}
\tag{$D_{\lambda}$}
\label{eq: dual-problem}
\max_{z \in \mathbb{R}^m} \quad D_{\lambda}(z) := -f^*(-B^Tz) - \lambda p^*(z/\lambda).
\end{equation}
Here, $f^*$ and $p^*$ are the conjugate of $f$ and $p$, respectively. Denote the solution set to \eqref{eq: two-block-structured-reform} as $\Omega_{\lambda}$. The Karush-Kuhn-Tucker (KKT) conditions imply that $(x^*, y^*) \in \Omega_{\lambda}$ if and only if there exists $ {z}^* \in \mathbb{R}^m$ such that
\begin{equation}\label{eq: KKT-two-block-structured-reform}\tag{KKT}
\left\{
\begin{array}{l}
\nabla f(x^*) + B^T z^* = 0,\\
z^* \in \lambda \partial p(y^*),\\
Bx^* - y^* = 0.
\end{array}
\right.
\end{equation}
For any given triplet $(x, y, z) \in \mathbb{R}^n \times \mathbb{R}^m \times \mathbb{R}^m$, we define  the KKT residual function for problem \eqref{eq: two-block-structured-reform} as:
\begin{equation}\label{eq: KKT-residual}
R_{\lambda}(x, y, z) := \left(
\begin{array}{c}
\nabla f(x) + B^Tz\\
y - {\rm Prox}_{\lambda p}(y + z)\\
Bx - y
\end{array}
\right).
\end{equation}
We know that $(x^*, y^*) \in \Omega_{\lambda}$ if and only if there exists $z^* \in \mathbb{R}^m$ such that
\[
R_{\lambda}(x^*, y^*, z^*) = 0.
\]
In this paper, we make the following two mild assumptions.
\begin{assumption}
\label{assumption: nonempty-optimal}
For any given $\lambda > 0$, the optimal solution set $\Omega_{\lambda}$ to the optimization problem \eqref{eq: two-block-structured-reform} is non-empty and compact.
\end{assumption}

\begin{assumption}
\label{assumption: uniqueness}
For any given $\lambda > 0$ and $y \in \mathbb{R}^m$, we define
\[
I^c:= \{
i \in [m] \mid y_i \not= 0
\},
\]
if $I^c \neq \emptyset$, then $(\partial(\lambda p(y)))_{I^c}$ is a singleton.
\end{assumption}
\begin{remark}
Here, we make some remarks on  Assumption \ref{assumption: uniqueness}. For most of the common used regularizers, such as lasso \cite{tibshirani1996regression}, group lasso \cite{yuan2006model}, exclusive lasso \cite{zhou2010exclusive}, the Assumption \ref{assumption: uniqueness} is satisfied. Let us take the lasso regularizer as an example. If $p(y) = \|y\|_1$, then,
$$
(\partial(\lambda p(y)))_i = \left\{
\begin{array}{ll}
\lambda{\rm sign}(y_i) & \mbox{if $y_i \neq 0$},\\
\mbox{$[-\lambda, \lambda]$}  & \mbox{if $y_i = 0$}.
\end{array}
\right.
$$
Thus, we know that for any $y \neq 0$,
\[
(\partial(\lambda p(y)))_{I^c} = (\lambda{\rm sign}(y))_{I^c}
\]
is a singleton. Here ${\rm sign}(\cdot)$ is the signum function.
\end{remark}
 When the matrix $B$ is the identity mapping and $p$ is a regularizer that can induce sparsity, it has been demonstrated in \cite{lin2020adaptive} that we can substantially reduce the dimension of the problem
\eqref{eq: two-block-structured-reform} by applying the adaptive sieving technique. However, it is not clear whether a similar idea could benefit those models whose solutions are not sparse but have some special structures. In this paper, we will give a positive answer to this question in the following sections.

The theme of this paper is to design a technique which can reduce the dimension of a class of optimization problem \eqref{eq: two-block-structured-reform} with structured sparsity by exploring the intrinsic structure of the problem in an explicit way. Readers will see shortly that the key idea behind is quite simple but a rigorous realization of this simple idea is highly non-trivial.

We first introduce our principal idea in a general way, then, we will propose a technique called adaptive sieving (AS) to rigorously implement the idea. We fix the parameter $\lambda$ in \eqref{eq: two-block-structured-reform} for now. For a given index set $I \subseteq [m]$, if there is some prior knowledge for us to believe that $y_I = 0$, which we call the structured sparsity, then it is natural for us to consider the following constrained optimization problem generated by the index set $I$:
\begin{equation}\label{eq:AS-reduced}
\tag{$P_{\lambda}(I)$}
\begin{array}{ll}
\min_{x \in \mathbb{R}^n, y \in \mathbb{R}^m} & f(x) + \lambda p(y)\\
& Bx - y = 0,\\
& y_I = 0.
\end{array}
\end{equation}
 We denote this problem as \eqref{eq:AS-reduced} to indicate its dependence on the index set $I$, we will denote the index set $[m] \backslash I$, which is the complement of $I$, as $I^c$.

Our principal idea is to obtain a solution to the original optimization problem \eqref{eq: two-block-structured-reform} by solving a sequence of subproblems with lower dimension, which are induced by \eqref{eq:AS-reduced}. The key for a successful realization of this principal idea for solving  \eqref{eq: two-block-structured-reform} depends on answering the following questions:
\begin{itemize}
    \item[Q1:] For a given index set $I \in [m]$, how to effectively reduce the dimension of \eqref{eq: two-block-structured-reform} based on \eqref{eq:AS-reduced}?
    \item[Q2:] If we can solve \eqref{eq:AS-reduced} to obtain a solution pair $(\bar{x}, \bar{y})$, which is \textbf{not} yet a solution to \eqref{eq: two-block-structured-reform}, can we guarantee that we can update the index set $I$ to construct a new reduced problem in the form of \eqref{eq:AS-reduced}?
    \item[Q3:] If the solution pair $(\bar{x}, \bar{y})$, which is obtained by solving \eqref{eq:AS-reduced} , is indeed a solution to \eqref{eq: two-block-structured-reform}, can we certify its optimality and stop the whole procedure?
    \item[Q4:] Is the proposed technique robust to the inexactness of the obtained solution pair? In other words, if we can only obtain an inexact solution of \eqref{eq: AS-small} (defined in Section 2.1) under a given tolerance $\epsilon > 0$, can we obtain a solution of \eqref{eq: two-block-structured-reform} under the tolerance $O(\epsilon)$?
    \item[Q5:] Is it possible for the proposed technique to be solver independent? In other words, the technique could be applied to any algorithms that can solve \eqref{eq: AS-small} inexactly under a given tolerance.
\end{itemize}
\begin{remark}
We make some remarks before we describe the proposed AS technique.
\begin{itemize}
    \item[1.] Although designing an efficient and
    convergent algorithm for solving \eqref{eq:AS-reduced} is also an important task, it is not the main purpose of this paper. Actually, as one may see later, our proposed AS technique is solver independent. There are also existing algorithms which can solve \eqref{eq: AS-small} to a moderate accuracy \cite{chi2015splitting,sun2021convex,yuan2018efficient}.
    \item[2.] Although it seems unnecessary to raise the question Q3 at the first glance, it is actually essential for applying any dimension reduction technique to solve \eqref{eq: two-block-structured-reform} based on \eqref{eq:AS-reduced}. In order to check the optimality of the solution pair $(\bar{x}, \bar{y})$,  we need to construct the corresponding dual solution and check the corresponding KKT condition \eqref{eq: KKT-two-block-structured-reform}. This is highly non-trivial since the dual solutions are not unique if structured sparsity exists. This is also one of the main difficulties for applying the AS technique to problem \eqref{eq: two-block-structured-reform} with structured sparsity as compared to \cite{lin2020adaptive}.
    \item[3.] The robustness mentioned in Q4 is also very important since the best we can expect in general is to obtain an inexact solution to \eqref{eq: AS-small}.
\end{itemize}
\end{remark}

Now, we start to give the details of our realization of the principal idea.
\subsection{A Dimension Reduction Technique for \eqref{eq: two-block-structured-reform} Based on \eqref{eq:AS-reduced}}

We first show how we can reduce the dimension of the variables $x$ and $y$ simultaneously for the problem \eqref{eq: two-block-structured-reform} based on the constrained optimization problem \eqref{eq:AS-reduced}, which is one of the core ideas of this paper. These details also answer the question Q1.

Assume that the rank of $B_{I:}$ is $r > 0$.  Then there exists three index sets $\alpha$, $\beta$ and $\gamma$ with $|\gamma| = r$, that forms a partition of $[n]$, such that $B_{I\beta} = 0$ and $B_{I\gamma}$ has full column rank. Here, we also assume that the index set $\alpha$ is nonempty; otherwise, we must have
\[
x_{\gamma} = 0.
\]
Since $B_{I\gamma}$ has full column rank,
there is a unique $|\gamma| \times |\alpha|$ matrix $M_{\gamma\alpha}$,\footnote{Here, we abuse the notation a little bit to indicate the dependence of $M_{\gamma\alpha}$ on the index sets $\alpha$ and $\gamma$. The uniqueness is in the sense of a given partition.} such that
\begin{equation}\label{eq: Mga}
B_{I\alpha} + B_{I\gamma}M_{\gamma\alpha} = 0.
\end{equation}
Then, we can eliminate $x_{\gamma}$ by the constraints of \eqref{eq:AS-reduced} as
\[
x_{\gamma} = M_{\gamma\alpha}x_{\alpha}.
\]
Define:
$$\varphi(x_{\alpha}, x_{\beta}) = f(\dot{x}), \quad q(y_{I^c}) = p(\dot{y}),$$
where
\[
\dot{x}_{\alpha} = x_{\alpha}, \quad \dot{x}_{\beta} = x_{\beta}, \quad \dot{x}_{\gamma} = M_{\gamma\alpha}x_{\alpha},
\]
and
\[
\dot{y}_i = \left\{
\begin{array}{ll}
y_i, & \mbox{if $i \in I^c$},\\
0 & \mbox{if $i \in I$}.
\end{array}
\right.
\]
It is not difficult to realize that we can solve problem \eqref{eq:AS-reduced} via solving the following reduced optimization problem:
\begin{equation}\label{eq: AS-small}
\tag{$RP_{\lambda}(I)$}
\begin{array}{ll}
\min_{x_{\alpha} \in \mathbb{R}^{|\alpha|}, x_{\beta} \in \mathbb{R}^{|\beta|}, y_{I^c} \in \mathbb{R}^{|I^c|}} & \varphi(x_{\alpha}, x_{\beta}) + \lambda q(y_{I^c})\\
{\rm s.t.} & (B_{I^c\alpha} + B_{I^c\gamma}M_{\gamma\alpha}) x_{\alpha} + B_{I^c\beta}x_{\beta} - y_{I^c} = 0.
\end{array}
\end{equation}
 The Lagrange function corresponding to \eqref{eq: AS-small} is given by
\[
l(x_{\alpha}, x_{\beta}, y_{I^c}, \xi) = \varphi(x_{\alpha}, x_{\beta}) + \lambda q(y_{I^c}) + \langle \xi, (B_{I^c\alpha} + B_{I^c\gamma}M_{\gamma\alpha}) x_{\alpha} + B_{I^c\beta}x_{\beta} - y_{I^c} \rangle,
\]
where $\xi \in \mathbb{R}^{|I^c|}$ is the Lagrange multiplier.

Now, if we solve \eqref{eq: AS-small} and obtain a solution $(\hat{x}_{\alpha}, \hat{x}_{\beta}, \hat{y}_{I^c})$, then, there exists a $\hat{\xi}$ that satisfies the following KKT condition:
\begin{equation}
\label{eq: KKT-RPI}
\left\{
\begin{array}{l}
(\nabla f(\hat{x}))_{\alpha} + M_{\gamma\alpha}^T(\nabla f(\hat{x}))_{\gamma} + (B_{I^c\alpha} + B_{I^c\gamma}M_{\gamma\alpha})^T\hat{\xi} = 0,\\[3pt]
(\nabla f(\hat{x}))_{\beta} + B_{I^c\beta}^T\hat{\xi} = 0,\quad \hat{\xi} \in (\partial (\lambda p(\hat{y})))_{I^c},\\[3pt]
(B_{I^c\alpha} + B_{I^c\gamma}M_{\gamma\alpha}) \hat{x}_{\alpha} + B_{I^c\beta}\hat{x}_{\beta}  - \hat{y}_{I^c} = 0,
\end{array}
\right.
\end{equation}
where $\hat{x}$ and $\hat{y}$ are defined as
\[
\hat{x}_{\alpha} = \hat{x}_{\alpha}, \quad \hat{x}_{\beta} = \hat{x}_{\beta}, \quad \hat{x}_{\gamma} = M_{\gamma\alpha}\hat{x}_{\alpha} \]
and
\[
\hat{y}_{I^c} = \hat{y}_{I^c}, \quad \hat{y}_I = 0,
\]
respectively.
Then $(\bar{x}, \bar{y})$, which is constructed by
\begin{equation}
\label{eq: recover_primal}
\left\{
\begin{array}{l}
\bar{x}_{\alpha} = \hat{x}_{\alpha},\quad \bar{x}_{\beta} = \hat{x}_{\beta},\quad \bar{x}_{\gamma} = M_{\gamma\alpha}\hat{x}_{\alpha}, \\
\bar{y}_{I^c} = \hat{y}_{I^c}, \quad \bar{y}_I = 0
\end{array}
\right.
\end{equation}
is a solution to problem \eqref{eq:AS-reduced}. Thus, in order to obtain a solution to \eqref{eq:AS-reduced}, we only need to solve a corresponding reduced problem \eqref{eq: AS-small} whose dimension
can be much smaller.
\begin{remark}
  We make some remarks to close this subsection.
  \begin{itemize}
      \item[1.] We reduce the dimension of the problem from $\mathbb{R}^n \times \mathbb{R}^m$ to $\mathbb{R}^{n - |\gamma|} \times \mathbb{R}^{m - |I|}$, which can be a substantial reduction. For example, if the solution of \eqref{eq: two-block-structured-reform} is indeed sparse (this is an intrinsic property since we can obtain a sparse solution in general for large $\lambda$), then $|I|$ is close to $m$ and $|\gamma|$ is close to $n$ simultaneously.
      \item[2.] In many real applications (e.g. convex clustering), we can identify the index set $\alpha$, $\beta$, $\gamma$ and construct the matrix $M_{\gamma\alpha}$ at a low cost. Also, since the linear map $B$ is designed to encode some structures of the solution, it is usually very sparse.
  \end{itemize}
\end{remark}

\subsection{An Adaptive Sieving Technique for \eqref{eq: two-block-structured-reform} with a Fixed $\lambda > 0$}
Now, we move on to present the details of the AS technique. We fix the parameter $\lambda$ for now and we will generalize it to handle the case for a sequence of $\lambda > 0$ later. Also, for
simplicity, we first present the idea with the assumption that we can solve \eqref{eq: AS-small} exactly. The same idea will be generalized to the inexact setting without much difficulties later.

We first show how we can update the index set $I$ if the current obtained solution $(\bar{x}, \bar{y})$ via solving \eqref{eq:AS-reduced} is not an optimal solution to \eqref{eq: two-block-structured-reform}.
The key idea is to construct a corresponding dual variable pair $(\bar{u}, \bar{w}) \in \mathbb{R}^m \times \mathbb{R}^{|I|}$ which satisfies the following KKT condition for \eqref{eq:AS-reduced}:
\begin{equation}\label{eq: KKT-AS-reduced}
\left\{
\begin{array}{l}
(\nabla f(\bar{x}))_{\alpha} + B_{I\alpha}^T\bar{u}_I + B_{I^c\alpha}^T\bar{u}_{I^c} = 0, \\
(\nabla f(\bar{x}))_{\beta} + B_{I^c\beta}^T\bar{u}_{I^c} = 0,\quad (\nabla f(\bar{x}))_{\gamma} + B_{I\gamma}^T\bar{u}_I + B_{I^c\gamma}^T\bar{u}_{I^c} = 0,\\[3pt]
\bar{u}_I -\bar{w} \in \lambda(\partial p(\bar{y}))_I, \quad \bar{u}_{I^c} \in \lambda(\partial p(\bar{y}))_{I^c},\\[3pt]
B\bar{x} - \bar{y} = 0, \quad \bar{y}_I = 0.
\end{array}
\right.
\end{equation}

Since $(\bar{x}, \bar{y}) = (\hat{x}, \hat{y})$ and $(\hat{x}, \hat{y}, \hat{\xi})$ is a solution to \eqref{eq: KKT-RPI}, we must have
\[
B^T_{I^c\beta}\hat{\xi} = B^T_{I^c\beta}\bar{u}_{I^c}, \quad \hat{\xi} \in (\partial(\lambda p(\hat{y})))_{I^c}, \quad \bar{u}_{I^c} \in (\partial(\lambda p(\bar{y})))_{I^c}.
\]
Aggressively, we construct $\bar{u}_{I^c}$ as
\begin{equation}
\label{eq: u_Ic_def}
\bar{u}_{I^c} = \hat{\xi}.
\end{equation}
By the above construction of $\bar{u}_{I^c}$ and the equation \eqref{eq: Mga}, the first equation of \eqref{eq: KKT-AS-reduced} is implied by the third equation of \eqref{eq: KKT-AS-reduced} and the first equation of \eqref{eq: KKT-RPI}. Thus, we can construct the pair $(\bar{u}_I, \bar{w})$ via solving the following equations for $(u_I, w)$:
\begin{equation}
\label{eq: construct_u_w}
\left\{
\begin{array}{l}
(\nabla f(\bar{x}))_{\gamma} + B^T_{I\gamma}u_I + B^T_{I^c\gamma}\bar{u}_{I^c} = 0,\\
u_I - w \in (\partial(\lambda p(\bar{y})))_{I}.
\end{array}
\right.
\end{equation}
Since $w$ is an unconstrained variable,  for any $\hat{u}_I$ satisfying the first equation of \eqref{eq: construct_u_w}, there exists a $\hat{w}$ such that the second one is satisfied. However, realizing the fact that if there exists a $\tilde{u}_I$ such that $(\tilde{u}_I, 0)$ is a solution to \eqref{eq: construct_u_w}, then the current solution pair $(\bar{x}, \bar{y})$ is an optimal solution to \eqref{eq: two-block-structured-reform}. Thus, we propose to construct the pair $(\bar{u}_I, \bar{w})$ such that $\bar{w}$ has the minimum Euclidean norm. Since $B_{I\gamma}$ has full column rank, we can construct a particular solution to the first equation of \eqref{eq: construct_u_w} as
\begin{equation}
\label{eq: uI0}
(\bar{u}_I)_0 = -B_{I\gamma}(B^T_{I\gamma}B_{I\gamma})^{-1}((\nabla f(\bar{x}))_{\gamma} + B^T_{I^c\gamma}\bar{u}_{I^c}).
\end{equation}
Thus, all the solution to the first equation of \eqref{eq: construct_u_w} is given by
\[
u_I = (\bar{u}_I)_0 + d,
\]
where $d \in {\rm Null}(B^T_{I\gamma})$. In summary, we
construct the solution pair $(\bar{u}_I, \bar{w})$ as follows:
\begin{equation}
\label{eq: recover_uw}
\bar{u}_I = (\bar{u}_I)_0 + \bar{d}, \quad \bar{w} = \bar{u}_I - \Pi_{(\partial(\lambda p(\bar{y})))_I}(\bar{u}_I),
\end{equation}
where $\bar{d}$ is a solution to the following auxiliary optimization problem:
\begin{equation}
\label{eq: project_dbar}
\begin{array}{ll}
\min_{d \in \mathbb{R}^{|I|}} & \frac{1}{2}\|((\bar{u}_I)_0 + d) - \Pi_{(\partial(\lambda p(\bar{y})))_I}((\bar{u}_I)_0 + d)\|^2 \\
{\rm s.t.} & d \in {\rm Null}(B^T_{I\gamma}).
\end{array}
\end{equation}
Up to this point, we have completed the construction of a dual solution pair $(\bar{u}, \bar{w})$. We show the nice properties of the constructed $(\bar{u}, \bar{w})$ in Theorem \ref{Theorem: answers to Q2} and Theorem \ref{thm: finite-convergence}.
\begin{theorem}
\label{Theorem: answers to Q2}
Assume that $(\hat{x}_{\alpha}, \hat{x}_{\beta}, \hat{y}_{I^c})$ is an optimal  solution to the following optimization problem:
\begin{equation}
\label{eq: inexactRPI}
\begin{array}{ll}
\min_{x_{\alpha} \in \mathbb{R}^{|\alpha|}, x_{\beta} \in \mathbb{R}^{|\beta|}, y_{I^c} \in \mathbb{R}^{|I^c|}} & \varphi(x_{\alpha}, x_{\beta}) + \lambda q(y_{I^c}) + \langle x_\alpha, \hat{\delta}_{1} \rangle + \langle x_{\beta}, \hat{\delta}_2 \rangle - \langle y_{I^c}, \hat{\delta}_3 \rangle \\
{\rm s.t.} & (B_{I^c\alpha} + B_{I^c\gamma}M_{\gamma\alpha}) x_{\alpha} + B_{I^c\beta}x_{\beta} - y_{I^c} = 0,
\end{array}
\end{equation}
and $\hat{\xi}$ is the corresponding Lagrange multiplier. Here, $(\hat{\delta}_1, \hat{\delta}_2, \hat{\delta}_3) \in \mathbb{R}^{|\alpha|} \times \mathbb{R}^{|\beta|} \times \mathbb{R}^{|I^c|}$ are given error terms satisfying $\|\hat{\delta}_1\| + \|\hat{\delta}_2\| + \|\hat{\delta}_3\| \leq \epsilon$. Let $(\bar{x}, \bar{y}, \bar{u}_{I^c}, \bar{u}_{I}, \bar{w})$ be the solution that is constructed from \eqref{eq: recover_primal}, \eqref{eq: u_Ic_def}, and \eqref{eq: recover_uw}. Define  $J(\lambda)$ as follows:
\begin{equation}
\label{eq: Jlambda_def}
J(\lambda) := \{j \in I \mid \bar{u}_j \not\in (\partial(\lambda p(\bar{y})))_j\}.
\end{equation}
Then, $J(\lambda) \not= \emptyset$ if
\[
\|R_{\lambda}(\bar{x}, \bar{y}, \bar{u})\| > \epsilon.
\]
\end{theorem}
\begin{proof}
Since $(\hat{x}_{\alpha}, \hat{x}_{\beta}, \hat{y}_{I^c})$ is an optimal  solution to \eqref{eq: inexactRPI} and $\hat{\xi}$ is the corresponding Lagrange multiplier,   the following KKT system holds:
\begin{equation}
\label{eq: KKT-inexactRPI}
\left\{
\begin{array}{l}
(\nabla f(\hat{x}))_{\alpha} + M_{\gamma\alpha}^T(\nabla f(\hat{x}))_{\gamma} + (B_{I^c\alpha} + B_{I^c\gamma}M_{\gamma\alpha})^T\hat{\xi} + \hat{\delta}_1 = 0,\\[3pt]
(\nabla f(\hat{x}))_{\beta} + B_{I^c\beta}^T\hat{\xi} + \hat{\delta}_2 = 0,\\[3pt]
\hat{\xi} + \hat{\delta}_3 \in (\partial (\lambda p(\hat{y})))_{I^c},\\[3pt]
(B_{I^c\alpha} + B_{I^c\gamma}M_{\gamma\alpha}) \hat{x}_{\alpha} + B_{I^c\beta}\hat{x}_{\beta}  - \hat{y}_{I^c} = 0.
\end{array}
\right.
\end{equation}
By construction, $(\bar{x}, \bar{y}, \bar{u}_{I^c}, \bar{u}_{I}, \bar{w})$ is a solution to:
\begin{equation}\label{eq: KKT-inexactPI}
\left\{
\begin{array}{l}
(\nabla f(\bar{x}))_{\alpha} + B_{I\alpha}^T\bar{u}_I + B_{I^c\alpha}^T\bar{u}_{I^c} + \hat{\delta}_1 = 0, \\[3pt]
(\nabla f(\bar{x}))_{\beta} + B_{I^c\beta}^T\bar{u}_{I^c} + \hat{\delta}_2 = 0,\\[3pt]
(\nabla f(\bar{x}))_{\gamma} + B_{I\gamma}^T\bar{u}_I + B_{I^c\gamma}^T\bar{u}_{I^c} = 0,\\[3pt]
\bar{u}_I -\bar{w} \in \lambda(\partial p(\bar{y}))_I,\\[3pt]
\bar{u}_{I^c} + \hat{\delta}_3 \in \lambda(\partial p(\bar{y}))_{I^c},\\[3pt]
B\bar{x} - \bar{y} = 0, \quad \bar{y}_I = 0.
\end{array}
\right.
\end{equation}
Now, we prove that $J(\lambda) \not= \emptyset$ provided $\|R_{\lambda}(\bar{x}, \bar{y}, \bar{u})\| > \epsilon$. We prove it by contradiction. Assume that
\[
J(\lambda) = \emptyset.
\]
Then we have
\[
\bar{u} + \hat{\delta} \in \partial (\lambda p(\bar{y})),
\]
where $\hat{\delta} = (\hat{\delta}_I, \hat{\delta}_{I^c}) = (0, \hat{\delta}_3)$. This implies that
\[
\bar{y} - {\rm Prox}_{\lambda p}(\bar{y} + (\bar{u} + \hat{\delta})) = 0.
\]
Then,
\begin{equation}
\label{eq: Residual-control}
\begin{array}{lll}
\|R_{\lambda}(\bar{x}, \bar{y}, \bar{u})\| & = & \|(\nabla f(\bar{x}) + B^T\bar{u}, \bar{y} - {\rm Prox}_{\lambda p}(\bar{y} + \bar{u}), B\bar{x} - \bar{y})\|\\
&=& \|((-\hat{\delta}_1, -\hat{\delta}_2, 0), {\rm Prox}_{\lambda p}(\bar{y} + (\bar{u} + \hat{\delta})) - {\rm Prox}_{\lambda p}(\bar{y} + \bar{u}), 0)\| \\
& \leq & \|\hat{\delta}_1\| + \|\hat{\delta}_2\| + \|\hat{\delta}_3\|\\
& \leq & \epsilon.
\end{array}
\end{equation}
Here, we used the property that the proximal mapping is Lipchitiz continuous with modulus $1$. This is a contradiction. Thus $J(\lambda) \not = \emptyset$ and we proved the statement in the theorem.
\end{proof}
\begin{remark}
 We do not need to specify a priori error terms $\hat{\delta}_1,\hat{\delta}_2,\hat{\delta}_3$
 in Theorem \ref{Theorem: answers to Q2}.
 They should be interpreted as the errors incurred when  we solve the problem \eqref{eq: AS-small} inexactly with a given tolerance.
\end{remark}
 An important implication of Theorem \ref{Theorem: answers to Q2} is that, if the current obtained solution pair $(\bar{x}, \bar{y})$ is not an inexact optimal solution to \eqref{eq: two-block-structured-reform} under the given tolerance, we can update the index set $I$ by removing the identified violated index set $J(\lambda)$. This important implication motivates us to propose the adaptive sieving (AS) technique for \eqref{eq: two-block-structured-reform} with a given fixed $\lambda > 0$, which is presented in Algorithm \ref{alg:screening}.

\begin{algorithm}[!ht]\small
	\caption{Adaptive sieving for solving \eqref{eq: two-block-structured-reform} with a fixed $\lambda > 0$}
	\label{alg:screening}
	\begin{algorithmic}[1]
		\STATE \textbf{Input}: a given hyper-parameter $\lambda > 0$ and a given tolerance $\epsilon > 0$.
		\STATE \textbf{Output}: $(x^*(\lambda), y^*(\lambda), z^*(\lambda))$.
		\STATE \textbf{Initialization}: Generate an initial index set by a predefined initialization strategy: $I^0(\lambda) \subseteq [m]$.
		\FOR{$i = 0, 1, 2, \dots$}
		\STATE \textbf{1}.For the given index set $I^i(\lambda)$, construct the index partition $\{\alpha^i, \beta^i, \gamma^i\}$ and the corresponding $M_{\gamma^i\alpha^i}$.
		\STATE \textbf{2}. Apply any well designed algorithm to solve problem \eqref{eq: AS-small} with $\{I^{i}(\lambda), \alpha^i, \beta^i, \gamma^i, M_{\gamma^i\alpha^i}\}$ and obtain an inexact solution $( \hat{x}^i_{\alpha_i}, \hat{x}^i_{\beta_i}, \hat{y}^i_{(I^i)^c(\lambda)}, \hat{\xi}^i)$
        which satisfies the corresponding KKT system \eqref{eq: KKT-inexactRPI} with the latent error terms $(\hat{\delta}_1^i, \hat{\delta}_2^i, \hat{\delta}_3^i)$ such that $\|\hat{\delta}^i_1\| + \|\hat{\delta}^i_2\| + \|\hat{\delta}^i_3\| \leq \epsilon$.
        \STATE \textbf{3}. Recover a solution $(\bar{x}^i, \bar{y}^i, \bar{u}^i, \bar{w}^i)$ by the construction of \eqref{eq: recover_primal}, \eqref{eq: u_Ic_def} and \eqref{eq: recover_uw}, respectively.
        \IF{$\|R_{\lambda}(\bar{x}^i, \bar{y}^i, \bar{u}^i)\| \leq \epsilon$}
        \STATE Set $(x^*(\lambda), y^*(\lambda), z^*(\lambda)) = (\bar{x}^i, \bar{y}^i, \bar{u}^i)$.
        \STATE \textbf{break}.
        \ELSE
        \STATE Create $J^{i}(\lambda)$:
        \begin{equation}\label{eq: AS-checking}
        J^i(\lambda) = \{j \in I^i(\lambda) \mid \bar{u}^i_j \not\in \partial ( \lambda p(\bar{y}^i))_j \},
        \end{equation}
        \IF{$J^i(\lambda) \neq \emptyset$}
        \STATE Update $I^{i+1}(\lambda)$ as:
        \[
        I^{i+1}(\lambda) \leftarrow I^i(\lambda) \backslash J^{i}(\lambda).
        \]
        \ELSE
        \STATE Set $(x^*(\lambda), y^*(\lambda), z^*(\lambda)) = (\bar{x}^i, \bar{y}^i, \bar{u}^i)$.
        \STATE \textbf{break}.
        \ENDIF
        \ENDIF
		\ENDFOR
	\RETURN{$(x^*(\lambda), y^*(\lambda), z^*(\lambda))$.}
	\end{algorithmic}
\end{algorithm}

\begin{theorem}
\label{thm: finite-convergence}
For a given $\epsilon > 0$, with any well designed algorithm which can solve the reduced subproblem \eqref{eq: AS-small} to the given accuracy, Algorithm \ref{alg:screening} is guaranteed to converge in finite number of iterations. Moreover, the obtained pair $(x^*(\lambda), y^*(\lambda), z^*(\lambda))$ is a solution to \eqref{eq: two-block-structured-reform} in the sense that \[\|R_{\lambda}(x^*(\lambda), y^*(\lambda), z^*(\lambda))\| \leq \epsilon.\]
\end{theorem}
We omit the proof of Theorem \ref{thm: finite-convergence} here as it is a byproduct of Theorem \ref{Theorem: answers to Q2}
\begin{remark}
\label{remark: after-theorem2}
 We close this subsection by making some remarks here.
 \begin{itemize}
     \item[1.] The proposed AS technique is a practical implementation of the aforementioned principal idea, which is solver independent and answers Q1, Q2, Q4 and Q5 simultaneously.
     \item[2.] However, it may fail to answer the question Q3. The whole procedure described in Algorithm \ref{alg:screening}
     is not guaranteed to certify the optimality of a given solution pair $(\bar{x}, \bar{y})$, even if it is already optimal for ($P_\lambda$). The constructed $\bar{u}$ may not be the correct corresponding Lagrange multiplier. The main reason is because we have aggressively set $\bar{u}_{I^c} = \hat{\xi}$ in \eqref{eq: u_Ic_def}.
     \item[3.] Although Algorithm \ref{alg:screening} may fail to answer the question Q3 and it may need additional iterations to terminate the whole algorithm, the practical performance of Algorithm \ref{alg:screening} is actually quite promising. Readers can find the numerical performance in the numerical experiments section.
     \item[4.] In order to address the possible weakness of the construction of $\bar{u}$ mentioned in item 2, we will propose an enhanced AS technique which can answer all the five questions simultaneously in the next subsection.
 \end{itemize}
\end{remark}

%%%%%%%%%%%%%%%%%%%%%%%%
\subsection{An Enhanced Adaptive Sieving Technique}
Now, we introduce an enhanced adaptive sieving technique which can certify the optimality of the obtained pair $(\bar{x}, \bar{y})$ via solving the reduced subproblem \eqref{eq: AS-small} if it is optimal to \eqref{eq: two-block-structured-reform}. With the enhanced AS technique, we can potentially reduce the number of  sieving iterations of Algorithm \ref{alg:screening}.

The key idea is to deal with the issue we mentioned in Remark \ref{remark: after-theorem2}. Now, assume that $(\bar{x}, \bar{y})$ is an optimal solution to \eqref{eq:AS-reduced}, which could be recovered by \eqref{eq: recover_primal} with a solution of \eqref{eq: AS-small}. We can then define a new index set $\tilde{I}$ as follows:
\begin{equation}
\label{eq: tilde_I}
\tilde{I} := \{i \in [m] \mid \bar{y}_i = 0\}.
\end{equation}
By the construction, we have $I \subseteq \tilde{I}$. It is not difficult to see that $(\bar{x}, \bar{y})$ is actually an optimal  solution to the following constrained optimization problem:
\begin{equation}
\tag{$P_{\lambda}(\tilde{I})$}
\label{eq: PtildeI}
\begin{array}{ll}
\min_{x \in \mathbb{R}^n, y \in \mathbb{R}^m} & f(x) + \lambda p(y) \\
{\rm s.t.} & Bx - y = 0 ,\\
& y_{\tilde{I}} = 0.
\end{array}
\end{equation}
In a similar manner, we can define the index sets $\tilde{\alpha}$, $\tilde{\beta}$ and $\tilde{\gamma}$ with $|\tilde{\gamma}| = \tilde{r}$, which form a partition of $[n]$, such that $B_{\tilde{I}\tilde{\beta}} = 0$ and $B_{\tilde{I}\tilde{\gamma}}$ has full column rank. Again, we assume that $\tilde{\alpha} \neq \emptyset$. Thus, there exists a $M_{\tilde{\gamma}\tilde{\alpha}} \in \mathbb{R}^{|\tilde{\gamma}|\times|\tilde{\alpha}|}$ such that
\[
B_{\tilde{I}\tilde{\alpha}} + B_{\tilde{I}\tilde{\gamma}}M_{\tilde{\gamma}\tilde{\alpha}} = 0.
\]
Then, we can eliminate $x_{\tilde{\gamma}}$ by the constraints of \eqref{eq: PtildeI} as
\[
x_{\tilde{\gamma}} = M_{\tilde{\gamma}\tilde{\alpha}}x_{\tilde{\alpha}}.
\]
The Lagrangian function corresponding to \eqref{eq: PtildeI} is given by
\[
l(x, y, v, s) = f(x) + \lambda p(y) + \langle v, Bx - y \rangle + \langle s, y_{\tilde{I}} \rangle,
\]
where $v \in \mathbb{R}^m$ and $s \in \mathbb{R}^{|\tilde{I}|}$ are the Lagrange multipliers. For notational consistency, we denote $(\tilde{x}, \tilde{y}) = (\bar{x}, \bar{y})$. Since $(\tilde{x}, \tilde{y})$ is an optimal solution to \eqref{eq: PtildeI}, there exists $(\tilde{v}, \tilde{s})$ such that the following KKT condition  for \eqref{eq: PtildeI} is satisfied: \begin{equation}\label{eq: KKT-PtildeI}
\left\{
\begin{array}{l}
(\nabla f(\tilde{x}))_{\tilde{\alpha}} + B_{\tilde{I}\tilde{\alpha}}^T\tilde{v}_{\tilde{I}} + B_{\tilde{I}^c\tilde{\alpha}}^T\tilde{v}_{\tilde{I}^c} = 0, \quad (\nabla f(\tilde{x}))_{\tilde{\beta}} + B_{\tilde{I}^c\tilde{\beta}}^T\tilde{v}_{\tilde{I}^c} = 0,\\
(\nabla f(\tilde{x}))_{\tilde{\gamma}} + B_{\tilde{I}\tilde{\gamma}}^T\tilde{v}_{\tilde{I}} + B_{\tilde{I}^c\tilde{\gamma}}^T\tilde{v}_{\tilde{I}^c} = 0,\\[3pt]
\tilde{v}_{\tilde{I}} -\tilde{s} \in \lambda(\partial p(\tilde{y}))_{\tilde{I}}, \quad \tilde{v}_{\tilde{I}^c} \in \lambda(\partial p(\tilde{y}))_{\tilde{I}^c},\\[3pt]
B\tilde{x} - \tilde{y} = 0, \quad \tilde{y}_{\tilde{I}} = 0.
\end{array}
\right.
\end{equation}
On the other hand, we know that $(\tilde{x}_{\tilde{\alpha}}, \tilde{x}_{\tilde{\beta}}, \tilde{y}_{\tilde{I}^c})$ is an optimal solution to the following reduced problem corresponding to \eqref{eq: PtildeI}:
\begin{equation}
\tag{$RP_{\lambda}(\tilde{I})$}
\label{eq: RP_tildeI}
\begin{array}{ll}
\min_{x_{\tilde{\alpha}}\in \mathbb{R}^{|\tilde{\alpha}|}, x_{\tilde{\beta}} \in \mathbb{R}^{|\tilde{\beta}|}, y_{\tilde{I}^c}\in \mathbb{R}^{|\tilde{I}^c|}} & \tilde{\varphi}(x_{\tilde{\alpha}}, x_{\tilde{\beta}}) + \lambda \tilde{q}(y_{\tilde{I}^c}) \\
{\rm s.t.} & (B_{\tilde{I}^c\tilde{\alpha}} + B_{\tilde{I}^c\tilde{\gamma}}M_{\tilde{\gamma}\tilde{\alpha}}) x_{\tilde{\alpha}} + B_{\tilde{I}^c\tilde{\beta}}x_{\tilde{\beta}} - y_{\tilde{I}^c} = 0,
\end{array}
\end{equation}
where
\[
\tilde{\varphi}(x_{\tilde{\alpha}}, x_{\tilde{\beta}}) = f(\dot{x}), \quad
\tilde{q}(y_{\tilde{I}^c}) = p(\dot{y}).
\]
Here
\[
\dot{x}_{\tilde{\alpha}} = x_{\tilde{\alpha}}, \quad \dot{x}_{\tilde{\beta}} = x_{\tilde{\beta}}, \quad \dot{x}_{\tilde{\gamma}} = M_{\tilde{\gamma}\tilde{\alpha}}x_{\tilde{\alpha}},
\]
and
\[
\dot{y}_i = \left\{
\begin{array}{ll}
y_i, & \mbox{if $i \in \tilde{I}^c$},\\
0 & \mbox{if $i \in \tilde{I}$}.
\end{array}
\right.
\]
Since $(\tilde{x}_{\tilde{\alpha}}, \tilde{x}_{\tilde{\beta}}, \tilde{y}_{\tilde{I}^c})$ is an optimal solution to \eqref{eq: RP_tildeI}, there exists a $\tilde{\theta} \in \mathbb{R}^{|\tilde{I}^c|}$ such that the following KKT condition is satisfied:
\begin{equation}\label{eq: KKT-RPtildeI}
\left\{
\begin{array}{l}
(\nabla f(\tilde{x}))_{\tilde{\alpha}} + M_{\tilde{\gamma}\tilde{\alpha}}^T(\nabla f(\tilde{x}))_{\tilde{\gamma}} + (B_{\tilde{I}^c\tilde{\alpha}} + B_{\tilde{I}^c\tilde{\gamma}}M_{\tilde{\gamma}\tilde{\alpha}})^T\tilde{\theta} = 0,\\[3pt]
(\nabla f(\tilde{x}))_{\tilde{\beta}} + B_{\tilde{I}^c\tilde{\beta}}^T\tilde{\theta} = 0,\quad \tilde{\theta} \in \lambda(\partial p(\tilde{y}))_{\tilde{I}^c},\\[3pt]
(B_{\tilde{I}^c\tilde{\alpha}} + B_{\tilde{I}^c\tilde{\gamma}}M_{\tilde{\gamma}\tilde{\alpha}}) \tilde{x}_{\tilde{\alpha}} + B_{\tilde{I}^c\tilde{\beta}}\tilde{x}_{\tilde{\beta}}  - \tilde{y}_{\tilde{I}^c} = 0.
\end{array}
\right.
\end{equation}
Again, the key is to construct a dual pair $(\tilde{v}, \tilde{s})$ from the KKT system \eqref{eq: KKT-RPtildeI}
such that $(\tilde{x}, \tilde{y}, \tilde{v}, \tilde{s})$ is a solution to \eqref{eq: KKT-PtildeI}. Fortunately, by Assumption \ref{assumption: uniqueness} and the fact
$\tilde{I}^c = \{ i\in [m] \mid \tilde{y}_i\not = 0\}$,
%$\tilde{y}_{\tilde{I}^c} \not= 0$,
we have
\begin{eqnarray}
\tilde{v}_{\tilde{I}^c} = (\partial(\lambda p(\tilde{y})))_{\tilde{I}^c} = \tilde{\theta}.
\label{eq:theta}
\end{eqnarray}
Thus, by the uniqueness of $\tilde{v}_{\tilde{I}^c}$, the second equation of \eqref{eq: KKT-PtildeI} must be satisfied.

\begin{comment}
Second, it is not difficult to show that the first equation of \eqref{eq: KKT-PtildeI} could be implied by the third equation of \eqref{eq: KKT-PtildeI} and the first equation of \eqref{eq: KKT-RPtildeI}. Thus, we design to construct a pair $(\tilde{v}_{\tilde{I}}, \tilde{s})$ that satisfies the following equations for $(v_{\tilde{I}}, s)$:
\begin{equation}
\label{eq: construct_v_s}
\left\{
\begin{array}{l}
(\nabla f(\tilde{x}))_{\tilde{\gamma}} + B_{\tilde{I}\tilde{\gamma}}^T v_{\tilde{I}} + B_{\tilde{I}^c\tilde{\gamma}}^T\tilde{v}_{\tilde{I}^c} = 0,\\[3pt]
v_{\tilde{I}} - s \in (\partial (\lambda p(\tilde{y})))_{\tilde{I}}.
\end{array}
\right.
\end{equation}
Since $B_{\tilde{I}\tilde{\gamma}}$ has full column rank, we can construct a particular solution to the first equation of \eqref{eq: construct_v_s} as
\begin{equation}
\label{eq: v0}
(\tilde{v}_{\tilde{I}})_0 = -B_{\tilde{I}\tilde{\gamma}}(B_{\tilde{I}\tilde{\gamma}}^TB_{\tilde{I}\tilde{\gamma}})^{-1}(\nabla_{\tilde{\gamma}}f(\tilde{x}) + B_{\tilde{I}^c\tilde{\gamma}}^T\tilde{v}_{\tilde{I}^c}).
\end{equation}
Thus, the solution of the first equation of \eqref{eq: construct_v_s} is given by
\[
\tilde{v}_{\tilde{I}} = (\tilde{v}_{\tilde{I}})_0 + d,
\]
where $d \in {\rm Null}(B_{\tilde{I}\tilde{\gamma}}^T)$.

In order to keep the consistency that $(\tilde{v}_{\tilde{I}}, 0)$ is a solution to \eqref{eq: construct_v_s} if $(\tilde{x}, \tilde{y})$ is indeed an optimal solution to \eqref{eq: two-block-structured-reform},
\end{comment}
Similarly, we construct $(\tilde{v}_{\tilde{I}}, \tilde{s})$ as follows:
\begin{equation}
 \label{eq: recover_v_s}
 \tilde{v}_{\tilde{I}} = (\tilde{v}_{\tilde{I}})_0 + \tilde{d}, \quad  \tilde{s} = \tilde{v}_{\tilde{I}} - \Pi_{(\partial (\lambda p(\bar{y})))_{\tilde{I}}}(\tilde{v}_{\tilde{I}}),
 \end{equation}
 where
 \[
 (\tilde{v}_{\tilde{I}})_0 = -B_{\tilde{I}\tilde{\gamma}}(B_{\tilde{I}\tilde{\gamma}}^TB_{\tilde{I}\tilde{\gamma}})^{-1}\big((\nabla f(\tilde{x}))_{\tilde{\gamma}} + B_{\tilde{I}^c\tilde{\gamma}}^T\tilde{v}_{\tilde{I}^c}\big)
 \]
 and $\tilde{d}$ is an optimal solution to the following auxiliary optimization problem:
\begin{equation}\label{eq: projection_v_s}
\begin{array}{ll}
\min_{d \in \mathbb{R}^{|\tilde{I}|}} & \frac{1}{2}\|((\tilde{v}_{\tilde{I}})_0 + d) - \Pi_{(\partial (\lambda p(\tilde{y})))_{\tilde{I}}}((\tilde{v}_{\tilde{I}})_0 + d)\|^2\\[3pt]
{\rm s.t.} & d \in {\rm Null}(B_{\tilde{I}\tilde{\gamma}}^T).
\end{array}
\end{equation}
For the above constructed $(\tilde{v},\tilde{s})$, it has a
nice property to be summarized in the following theorem. It shows that the constructed dual variable $\tilde{v}$ can certify the optimality of $\tilde{x}$.
\begin{theorem}
\label{thm: early-stop}
For a given $\epsilon > 0$, if the current obtained solution $\tilde{x}$
by solving \eqref{eq: AS-small} is an optimal solution to the following perturbed optimization problem
\begin{equation}
\label{eq: inexact-Plambda}
\begin{array}{ll}
\min_{x \in \mathbb{R}^n} & f(x) + \lambda p(Bx) + \langle x, \tilde{\delta} \rangle,
\end{array}
\end{equation}
where $\tilde{\delta} \in \mathbb{R}^n$ is a latent error vector such that $\|\tilde{\delta}\| \leq \frac{\epsilon}{1 + 2L_{\tilde{\gamma}}}$, with $L_{\tilde{\gamma}} = \|B_{\tilde{I}\tilde{\gamma}}(B^T_{\tilde{I}\tilde{\gamma}}B_{\tilde{I}\tilde{\gamma}})^{-1}\|$. Then, we must have
\[
\|R_{\lambda}(\tilde{x}, \tilde{y}, \tilde{v})\| \leq \epsilon,
\]
where $\tilde{y} = B\tilde{x}$ and $\tilde{v}$ is constructed in \eqref{eq:theta}, \eqref{eq: recover_v_s} and \eqref{eq: projection_v_s}. Thus we can certify the optimality of $\tilde{x}$.
\end{theorem}
\begin{comment}
\noindent\myfbox{I find the logical sequence in presenting this theorem a bit strange, but I am not sure what is the best way to present it. Logically, it seems that one should start as follows: If $\norm{\tilde{s}} \leq \epsilon$, then $\norm{R_\lambda(\tilde{x},\tilde{y},\tilde{v})} \leq \epsilon$ and
$(\tilde{x},\tilde{y})$ is an optimal solution to the perturbed problem (26) with $\tilde{\delta}$ depending on $\tilde{s}$.}
\end{comment}
\begin{proof}
If $\tilde{x}$ is an optimal solution to \eqref{eq: inexact-Plambda}, then $(\tilde{x}, \tilde{y})$ is an optimal solution to
\begin{equation}
\label{eq: inexact-Plambda-reform}
\begin{array}{ll}
\min_{x \in \mathbb{R}^n, y \in \mathbb{R}^m} & f(x) + \lambda p(y) + \langle x, \tilde{\delta} \rangle\\
{\rm s.t.} & Bx - y = 0.
\end{array}
\end{equation}
Then, there exists a $\tilde{z} \in \mathbb{R}^m$ which satisfies the following KKT system:
\begin{equation}
\label{eq: KKT-inexact Plambda}
\left\{
\begin{array}{l}
(\nabla f(\tilde{x}))_{\tilde{\alpha}} + B_{\tilde{I}\tilde{\alpha}}^T\tilde{z}_{\tilde{I}} + B_{\tilde{I}^c\tilde{\alpha}}^T\tilde{z}_{\tilde{I}^c} + \tilde{\delta}_{\tilde{\alpha}} = 0,\\[3pt]
(\nabla f(\tilde{x}))_{\tilde{\beta}} + B^T_{\tilde{I}^c\tilde{\beta}}\tilde{z}_{\tilde{I}^c} + \tilde{\delta}_{\tilde{\beta}} = 0,\\[3pt]
(\nabla f(\tilde{x}))_{\tilde{\gamma}} + B_{\tilde{I}\tilde{\gamma}}^T\tilde{z}_{\tilde{I}} + B_{\tilde{I}^c\tilde{\gamma}}^T\tilde{z}_{\tilde{I}^c} + \tilde{\delta}_{\tilde{\gamma}} = 0,\\[3pt]
\tilde{z} \in \partial(\lambda p(\tilde{y})),\\
B\tilde{x} - \tilde{y} = 0.
\end{array}
\right.
\end{equation}
By Assumption \ref{assumption: uniqueness} and the fact
$\tilde{I}^c = \{ i\in [m] \mid \tilde{y}_i\not = 0\}$, $(\partial(\lambda p)(\tilde{y}))_{\tilde{I}^c}$ is a singleton. Thus we must have
\[
\tilde{z}_{\tilde{I}^c} = \tilde{v}_{\tilde{I}^c}.
\]
Therefore, $\bar{v}_{\tilde{I}} = \tilde{z}_{\tilde{I}} - B_{\tilde{I}\tilde{\gamma}}(B^T_{\tilde{I}\tilde{\gamma}}B_{\tilde{I}\tilde{\gamma}})^{-1}\tilde{\delta}_{\tilde{\gamma}}$ is a solution to
\[
(\nabla f(\tilde{x}))_{\tilde{\gamma}} + B_{\tilde{I}\tilde{\gamma}}^Tv_{\tilde{I}} + B_{\tilde{I}^c\tilde{\gamma}}^T\tilde{v}_{\tilde{I}^c} = 0.
\]
Since $\tilde{v}_{\tilde{I}}$ is a solution to \eqref{eq: projection_v_s}, we have
\[
\begin{array}{lll}
\|\tilde{s}\| &=& \|\tilde{v}_{\tilde{I}} - \Pi_{(\partial(\lambda p(\tilde{y})))_{\tilde{I}}}(\tilde{v}_{\tilde{I}})\|\\
& \leq & \|\bar{v}_{\tilde{I}} - \Pi_{(\partial(\lambda p(\tilde{y})))_{\tilde{I}}}(\bar{v}_{\tilde{I}})\|\\[3pt]
&=& \|(\tilde{z}_{\tilde{I}} - B_{\tilde{I}\tilde{\gamma}}(B^T_{\tilde{I}\tilde{\gamma}}B_{\tilde{I}\tilde{\gamma}})^{-1}\tilde{\delta}_{\tilde{\gamma}}) - \Pi_{(\partial(\lambda p(\tilde{y})))_{\tilde{I}}}(\tilde{z}_{\tilde{I}} - B_{\tilde{I}\tilde{\gamma}}(B^T_{\tilde{I}\tilde{\gamma}}B_{\tilde{I}\tilde{\gamma}})^{-1}\tilde{\delta}_{\tilde{\gamma}})\|\\[3pt]
&=&\|- B_{\tilde{I}\tilde{\gamma}}(B^T_{\tilde{I}\tilde{\gamma}}B_{\tilde{I}\tilde{\gamma}})^{-1}\tilde{\delta}_{\tilde{\gamma}} + (\Pi_{(\partial(\lambda p(\tilde{y})))_{\tilde{I}}}(\tilde{z}_{\tilde{I}}) - \Pi_{(\partial(\lambda p(\tilde{y})))_{\tilde{I}}}(\tilde{z}_{\tilde{I}} - B_{\tilde{I}\tilde{\gamma}}(B^T_{\tilde{I}\tilde{\gamma}}B_{\tilde{I}\tilde{\gamma}})^{-1}\tilde{\delta}_{\tilde{\gamma}}))\|\\[3pt]
&\leq& 2\|B_{\tilde{I}\tilde{\gamma}}(B^T_{\tilde{I}\tilde{\gamma}}B_{\tilde{I}\tilde{\gamma}})^{-1}\tilde{\delta}_{\tilde{\gamma}}\|\\
& \leq & 2L_{\tilde{\gamma}}\|\tilde{\delta}_{\tilde{\gamma}}\|.
\end{array}
\]
On the other hand, by the construction, we know that $(\tilde{x}, \tilde{y}, \tilde{v}, \tilde{s})$ satisfies the following KKT system
\[
\left\{
\begin{array}{l}
(\nabla f(\tilde{x}))_{\tilde{\alpha}} + B_{\tilde{I}\tilde{\alpha}}^T\tilde{v}_{\tilde{I}} + B_{\tilde{I}^c\tilde{\alpha}}^T\tilde{v}_{\tilde{I}^c} + \tilde{\delta}_{\tilde{\alpha}} = 0,\\[3pt]
(\nabla f(\tilde{x}))_{\tilde{\beta}} + B^T_{\tilde{I}^c\tilde{\beta}}\tilde{v}_{\tilde{I}^c} + \tilde{\delta}_{\tilde{\beta}} = 0,\\[3pt]
(\nabla f(\tilde{x}))_{\tilde{\gamma}} + B_{\tilde{I}\tilde{\gamma}}^T\tilde{v}_{\tilde{I}} + B_{\tilde{I}^c\tilde{\gamma}}^T\tilde{v}_{\tilde{I}^c} = 0,\\[3pt]
\tilde{v}_{\tilde{I}} - \tilde{s} \in (\partial(\lambda p(\tilde{y})))_{\tilde{I}}, \quad \tilde{v}_{\tilde{I}^c} \in (\partial(\lambda p(\tilde{y})))_{\tilde{I}^c}.\\
B\tilde{x} - \tilde{y} = 0.
\end{array}
\right.
\]
Then, we have
\[
\begin{array}{lll}
\|R_{\lambda}(\tilde{x}, \tilde{y}, \tilde{v})\| & = & \|(\nabla f(\tilde{x}) + B^T\tilde{v}, \quad \tilde{y} - {\rm Prox}_{\lambda p}(\tilde{y} + \tilde{v}), \quad B\tilde{x} - \tilde{y})\|\\
&\leq& \|(\tilde{\delta}_{\tilde{\alpha}}, \tilde{\delta}_{\tilde{\beta}}, 0)\| + \|\tilde{s}\|\\
& \leq & \|\tilde{\delta}\| + 2L_{\tilde{\gamma}}\|\tilde{\delta}\|\\
& \leq & \epsilon.
\end{array}
\]
This completes the proof of the theorem.
\end{proof}

  Now, we present the enhanced AS technique in Algorithm \ref{alg:enhancedAS}.
\begin{algorithm}[!ht]\small
	\caption{An enhanced adaptive sieving for solving \eqref{eq: two-block-structured-reform} with a fixed $\lambda > 0$}
	\label{alg:enhancedAS}
	\begin{algorithmic}[1]
		\STATE \textbf{Input}: a given hyper-parameter $\lambda > 0$ and a given tolerance $\epsilon > 0$.
		\STATE \textbf{Output}: $(x^*(\lambda), y^*(\lambda), z^*(\lambda))$.
		\STATE \textbf{Initialization}: Generate an initial index set by a predefined initialization strategy: $I^0(\lambda) \subseteq [m]$.
		\FOR{$i = 0, 1, 2, \dots$}
		\STATE \textbf{1}.For the given index set $I^i(\lambda)$, construct the index partition $\{\alpha^i, \beta^i, \gamma^i\}$ and the corresponding $M_{\gamma^i\alpha^i}$.
		\STATE \textbf{2}. Apply any well designed algorithm to solve problem \eqref{eq: AS-small} with $\{I^{i}(\lambda), \alpha^i, \beta^i, \gamma^i, M_{\gamma^i\alpha^i}\}$ and obtain an inexact solution $( \hat{x}^i_{\alpha_i}, \hat{x}^i_{\beta_i}, \hat{y}^i_{(I^i)^c(\lambda)}, \hat{\xi}^i)$
        which satisfies the corresponding KKT system \eqref{eq: KKT-inexactRPI} with the latent error terms $(\hat{\delta}_1^i, \hat{\delta}_2^i, \hat{\delta}_3^i)$ such that $\|\hat{\delta}^i_1\| + \|\hat{\delta}^i_2\| + \|\hat{\delta}^i_3\| \leq \epsilon$.
        \STATE \textbf{3}. Recover a solution $(\bar{x}^i, \bar{y}^i)$ by \eqref{eq: recover_primal}.
        \IF{$i > 1$ and $|F_{\lambda}(\bar{x}^i) - F_{\lambda}(\bar{x}^{i-1})| \leq \epsilon$}
        \STATE Define $(\tilde{x}^i, \tilde{y}^i) = (\bar{x}^i, B\bar{x}^i)$ and $\tilde{I}^i = \{i \in [m] \mid \tilde{y}_i = 0\}$. Construct $\{\tilde{\alpha}^i, \tilde{\beta}^i, \tilde{\gamma}^i, M_{\tilde{\gamma}^i\tilde{\alpha}^i}
        \}$.
        \STATE Construct $(\tilde{v}^i, \tilde{s}^i)$ by \eqref{eq:theta}, \eqref{eq: recover_v_s} and \eqref{eq: projection_v_s}.
        \IF{$\|R_{\lambda}(\tilde{x}^i, \tilde{y}^i, \tilde{v})\| \leq \epsilon$}
        \STATE Set $(x^*(\lambda), y^*(\lambda), z^*(\lambda)) = (\tilde{x}^i, \tilde{y}^i, \tilde{v}^i)$.
        \STATE \textbf{break}.
        \ENDIF
        \ENDIF
        \STATE \textbf{4}. Recover a pair $(\bar{u}^i, \bar{w}^i)$ by \eqref{eq: u_Ic_def} and \eqref{eq: recover_uw}, respectively.
        \IF{$\|R_{\lambda}(\bar{x}^i, \bar{y}^i, \bar{u}^i)\| \leq \epsilon$}
        \STATE Set $(x^*(\lambda), y^*(\lambda), z^*(\lambda)) = (\bar{x}^i, \bar{y}^i, \bar{u}^i)$.
        \STATE \textbf{break}.
        \ELSE
        \STATE Create $J^{i}(\lambda)$:
        \begin{equation}\label{eq: AS-checking-eas}
        J^i(\lambda) = \{j \in I^i(\lambda) \mid \bar{u}^i_j \not\in \partial ( \lambda p(\bar{y}^i))_j \}.
        \end{equation}
        \IF{$J^i(\lambda) \neq \emptyset$}
        \STATE Update $I^{i+1}(\lambda)$ as:
        \[
        I^{i+1}(\lambda) \leftarrow I^i(\lambda) \backslash J^{i}(\lambda).
        \]
        \ELSE
        \STATE Set $(x^*(\lambda), y^*(\lambda), z^*(\lambda)) = (\bar{x}^i, \bar{y}^i, \bar{u}^i)$.
        \STATE \textbf{break}.
        \ENDIF
        \ENDIF
		\ENDFOR
	\RETURN{$(x^*(\lambda), y^*(\lambda), z^*(\lambda))$.}
	\end{algorithmic}
\end{algorithm}
As a byproduct of Theorem \ref{Theorem: answers to Q2}, Theorem \ref{thm: finite-convergence} and Theorem \ref{thm: early-stop}, we have the following property.
% summarized in Theorem
%\ref{thm: finite-convergence-enhanced}.

\begin{theorem}
\label{thm: finite-convergence-enhanced}
For a given $\epsilon > 0$,  Algorithm  \ref{alg:enhancedAS} is guaranteed to converge in finite number of iterations. The number of sieving iterations of Algorithm \ref{alg:enhancedAS} is no more than that of Algorithm \ref{alg:screening}. Moreover, the obtained pair $(x^*(\lambda), y^*(\lambda), z^*(\lambda))$ is a solution to \eqref{eq: two-block-structured-reform} in the sense that
\[
\|R_{\lambda}(x^*(\lambda), y^*(\lambda), z^*(\lambda))\| \leq \epsilon.
\]
\end{theorem}
\begin{remark}
 We close this subsection by making some   remarks.
 \begin{itemize}
     \item[1.] The enhanced adaptive sieving technique described in Algorithm \ref{alg:enhancedAS} is a rigorous implementation of the aforementioned principal idea which simultaneously answers all the five questions we asked earlier in Section \ref{sec: AS-General}.
     \item[2.] A natural question is, why should we still perform
     the sieving based on  $\bar{u}$ instead of $\tilde{v}$ directly? Now we explain the reason. If we define
     \[
     \tilde{J}(\lambda) = \{ j \in \tilde{I} \mid \tilde{v}_j \not\in (\partial(\lambda p(\tilde{y})))_j\},
     \]
     and assuming that $\tilde{J}(\lambda) \neq \emptyset$, we cannot guarantee that $\tilde{J}(\lambda) \bigcap I \neq \emptyset$, which is required to update the index set $I$.
     \item[3.] The main idea for the enhanced algorithm is to
     certify the optimality of the current solution if
     it is an optimal solution of \eqref{eq: two-block-structured-reform}. Then we can stop the sieving procedure earlier, comparing to Algorithm \ref{alg:screening}. It is a natural idea that we only try to certify the optimality of the current obatined solution if it is the solution to \eqref{eq: two-block-structured-reform} with high probability. This is implied by the condition $|F_{\lambda}(\bar{x}^i) - F_{\lambda}(\bar{x}^{i-1})| < \epsilon$, which is used in Algorithm \ref{alg:enhancedAS}. The reason we use the difference of the consecutive function values instead of the solution vectors is because the optimal solutions of \eqref{eq: two-block-structured-reform} may not be unique, but they all have the same objective function value.
    \item[4.] In practice, the AS technique is sometimes better than the enhanced AS technique in terms of running time although the enhanced AS could potentially reduce the number of AS iterations. But of course, the enhanced AS technique is the one with a better theoretical guarantee.  Detailed empirical comparison of these techniques can be found in the numerical experiments.
 \end{itemize}
\end{remark}

\subsection{An Accelerated Proximal Gradient Algorithm for Dual Variables Recovery}
As aforementioned, the key step to recover the
dual variables and applying the AS technique is to recover $\bar{u}$ (or $\tilde{v}$) via solving the optimization problem \eqref{eq: project_dbar} (or \eqref{eq: projection_v_s}). In this paper, we adopt the accelerated proximal gradient (APG) algorithm \cite{beck2009fast,nesterov1983method} to solve it. Since the optimization problem \eqref{eq: projection_v_s} has the same form as \eqref{eq: project_dbar}, we use the problem \eqref{eq: project_dbar} as an example.

First of all, we could rewrite the constrained optimization problem \eqref{eq: project_dbar} equivalently as
\begin{equation}
\label{eq: proj_d_unconstrained}
\min_d h(d) + \delta_{{\rm Null}(B_{I\gamma}^T)}(d),
\end{equation}
where $h(d) = \frac{1}{2}\|((\bar{u}_{I})_0 + d) - \Pi_{\partial(\lambda p(\bar{y}))_I}((\bar{u}_I)_0 + d)\|^2$ and $\delta_{{\rm Null}(B_{I\gamma}^T)}(\cdot)$ is the indicator function of the Null space of $B_{I\gamma}^T$.

In order to apply the APG algorithm, we need to derive the proximal mapping of the indicator function $\delta_{{\rm Null}(B_{I\gamma}^T)}(\cdot)$, which is the projection operator onto the null space of $B_{I\gamma}^T$. Since $B_{I\gamma}^T$ is of full row rank, the projection of a given vector $a \in \mathbb{R}^{|I|}$ onto the null space of $B_{I\gamma}^T$ is computed by
\[
\Pi_{{\rm Null}(B_{I\gamma}^T)}(a) = (I - B_{I\gamma}(B_{I\gamma}^TB_{I\gamma})^{-1}B_{I\gamma}^T)a.
\]
On the other hand, the function $h(\cdot)$ is continuously differentiable and the gradient of $h(\cdot)$ is
\[
\nabla h(d) = ((\bar{u}_I)_0 + d) - \Pi_{\partial(\lambda p(\tilde{y}))_I}((\bar{u}_I)_0 + d) =  \Pi_{(\partial (\lambda p(\bar{y}))_I)^{\circ}}((\bar{u}_I)_0 + d).
\]
Here, $(\partial (\lambda p(\bar{y}))_I)^{\circ}$ is the polar of the closed convex set $\partial (\lambda p(\bar{y}))_I$ and the second equality comes from the Moreau identity \cite{moreau1965proximite}. Thus, $\nabla h(\cdot)$ is Lipschitz continuous with modulus $1$ \cite{zarantonello1971projections}. The APG algorithm for solving the optimization problem \eqref{eq: proj_d_unconstrained} is shown in Algorithm \ref{alg: apg_projection}.
\begin{algorithm}[H]\small
\caption{Accelerated proximal gradient algorithm for \eqref{eq: proj_d_unconstrained}}
\label{alg: apg_projection}
\begin{algorithmic}
\STATE {\textbf{Input}}: $\epsilon > 0$ and maxiter.
\STATE {\textbf{Output}}: $\bar{d}$.
\STATE {\textbf{Initialization}}: $L = 1$, $d^0 = 0$, $\hat{d}^1 = d^0$, $k = 0$ and $t_1 = 1$.
\WHILE {$k < {\rm maxiter}$}
\STATE $k = k + 1$,
\STATE $d^k = \Pi_{{\rm Null}(B_{I\gamma}^T)}(\hat{d}^k - \frac{1}{L}\nabla h(\hat{d}^k)),$
\IF {$\max(\|d^k - d^{k-1}\|, \|((\bar{u}_I)_0 + d^k) - \Pi_{\partial(\lambda p(\bar{y}))_I}((\bar{u}_I)_0 + d^k)\|) \leq \epsilon$}
\STATE break,
\ENDIF
\STATE $t_{k+1} = \frac{1 + \sqrt{1+4t_k^2}}{2}$,
\STATE $\hat{d}^{k+1} = d^k + (\frac{t_k - 1}{t_{k+1}})(d^k - d^{k-1})$.
\ENDWHILE
\STATE $\bar{d} = d^k$,
\RETURN{$\bar{d}$.}
\end{algorithmic}
\end{algorithm}
\noindent It is well known that the sequence $\{d^k\}$ generated by the APG algorithm have the following $O(1/k^2)$ complexity \cite{beck2009fast,nesterov1983method}.
\begin{theorem}
Let $\{d^k\}$ and $\{y^k\}$ be the sequences generated by Algorithm \ref{alg: apg_projection}. Then for any $k \geq 1$, we have
\begin{equation}\label{eq: APG_rate}
  h(d^k) - h(d^*) \leq \frac{2\|d^*\|^2}{(k+1)^2},
\end{equation}
where $d^*$ is any optimal solution to \eqref{eq: proj_d_unconstrained}.
\end{theorem}

\begin{remark}
  Although we need to solve an additional optimization problem \eqref{eq: proj_d_unconstrained} in order to apply the AS technique, the computational cost is affordable. Now, we explain the key insights behind. In the enhanced AS technique, if we do obtain an optimal solution of \eqref{eq: two-block-structured-reform} via solving the current subproblem \eqref{eq:AS-reduced}, then we must have $h(d^*) < \frac{\epsilon^2}{2}$ by Theorem \ref{thm: early-stop}. Moreover, $\|d^*\|$ must be relatively small. By the above complexity result, we could obtain an inexact solution to the problem \eqref{eq: proj_d_unconstrained} in several cheap iterations. On the other hand, if the objective function value of \eqref{eq: proj_d_unconstrained} is still large after several iterations (say $10$ iterations), we can terminate the algorithm since this phenomenon indicates that we have not yet obtained an optimal solution to the problem \eqref{eq: two-block-structured-reform}. In other words, the current index set $I$ is incorrect and we need to update it by removing
  violating indices. In short, although we need to solve an additional optimization problem, we only need to run APG for several iterations.

  The main computational cost for each iteration of APG is from two projections. For most of the commonly used regularizers $p$ (for example, $\ell_1$ norm, $\ell_2$ norm), the projection of a given vector onto the subdifferential set is very cheap. On the other hand, in order to compute the projection onto the null space of $B^T_{I\gamma}$, the main computational cost is from computing $(B^T_{I\gamma}B_{I\gamma})^{-1}$. However, as we mentioned earlier,the matrix $B$ is usually very sparse in many applications, the sparse Cholesky decomposition is not costly. Thus, the computational cost for one iteration of APG is affordable, even for large scale problems. This is also one of the main reason for us to adopt APG to solve the optimization problem \eqref{eq: proj_d_unconstrained}.
\end{remark}

\subsection{An Adaptive Sieving Technique for Solution Path}
It is not difficult for us to generalize Algorithm \ref{alg:enhancedAS} to obtain a solution path for problem \eqref{eq: two-block-structured-reform} with a sequence of
parameters $\lambda_1 > \lambda_2 > \cdots > \lambda_l > 0$. The key
idea is that, if we obtain a solution $(x^*(\lambda_i), y^*(\lambda_i), z^*(\lambda_i))$ for \eqref{eq: two-block-structured-reform} with $\lambda = \lambda_i$, then, we can initialize the index set $I^0_{i+1}$ in Algorithm \ref{alg:enhancedAS} for $\lambda = \lambda_{i+1}$ as
\begin{equation}
\label{eq: initial_I_i}
I^{0}_{i+1} := \{k \in [m] \mid  |(Bx^*(\lambda_i))_k| < \hat{\epsilon}\},
\end{equation}
where $\hat{\epsilon} > 0$ is a given tolerance. The algorithm for applying the AS technique (or the EAS technique) to generate a solution path is shown in Algorithm \ref{alg: screening-path}.
\begin{algorithm}[!ht]
\small
\caption{Generate solution path for \eqref{eq: two-block-structured-reform} with the AS technique (or the EAS technique)}
\label{alg: screening-path}
\begin{algorithmic}
\STATE {\textbf{Input}}: $\epsilon > 0$, $\hat{\epsilon} > 0$ and a sequence $\lambda_1 > \lambda_2 > \cdots > \lambda_l > 0$.
\STATE {\textbf{Output}}: A solution path for \eqref{eq: two-block-structured-reform}: $\{(x^*(\lambda_1), y^*(\lambda_1), z^*(\lambda_1)), \dots, (x^*(\lambda_l), y^*(\lambda_l), z^*(\lambda_l))\}$.
\STATE {\textbf{Initialization}}: Initialize index set $I^0(\lambda_1) \subseteq [m]$ by a predefined initialization strategy.
\FOR {$k = 1, 2, \dots, l$}
\STATE \textbf{Step 1}. Obtain $(x^*(\lambda_k), y^*(\lambda_k), z^*(\lambda_k))$ by calling Algorithm \ref{alg:screening} (or Algorithm \ref{alg:enhancedAS}) with $\{\lambda, \epsilon, I^0(\lambda)\} = \{\lambda_k, \epsilon, I^0(\lambda_k)\}$.
\IF{$k < l$}
\STATE \textbf{Step 2}. Define
\[
I^{0}(\lambda_{k+1}) := \{j \in [m] \mid  |(Bx^*(\lambda_k))_j| < \hat{\epsilon}\}.
\]
\ENDIF
\ENDFOR
\RETURN{$\{(x^*(\lambda_1), y^*(\lambda_1), z^*(\lambda_1)), \dots, (x^*(\lambda_l), y^*(\lambda_l), z^*(\lambda_l))\}$.}
\end{algorithmic}
\end{algorithm}

\section{Adaptive Sieving and Enhanced Adaptive Sieving Technique for Convex Clustering}
In this section, we will show how to apply the AS technique and the EAS technique on the convex clustering model \eqref{eq: convex-clustering}.

Denote $\mathcal{E} := \{(i, j) \mid w_{ij} > 0, 1 \leq i < j \leq n\}$. Then $\mathcal{G} = ([n], \mathcal{E})$ forms an undirected graph and the weighted convex clustering model \eqref{eq: convex-clustering} is equivalent to:
\begin{equation}\label{eq: convex-clustering-edge}
\min_{X \in \mathbb{R}^{d \times N}} \frac{1}{2}\sum_{i = 1}^{N}\|X_{:i} - A_{:i}\|_2^2 + \lambda \sum_{(i, j) \in \mathcal{E}} w_{ij}\|X_{:i} - X_{:j}\|_p.
\end{equation}
We enumerate the index pairs in $\mathcal{E}$ by the lexicographic order and denote by $l(i, j)$ for the pair $(i, j)$. Define the linear map $\mathcal{B}: \mathbb{R}^{d \times N} \to \mathbb{R}^{d \times |\mathcal{E}|}$ by
$$(\mathcal{B}(X))_{:,l(i, j)} = X_{:i} - X_{:j},$$
and the node-arc incidence matrix $J \in \mathbb{R}^{N \times|\mathcal{E}|}$ as
\begin{equation}\label{eq: define-node-arc}
J_{k, l(i,j)} = \left\{
\begin{array}{ll}
1, & \mbox{if $k = i$},\\
-1, & \mbox{if $k = j$},\\
0, & \mbox{otherwise}.
\end{array}
\right.
\end{equation}
Then, for any given $X \in \mathbb{R}^{d \times N}$ and $Z \in \mathbb{R}^{d \times |\mathcal{E}|}$, we have
\begin{equation}
\label{eq: map_B}
\mathcal{B}(X) = XJ, \quad \mathcal{B}^*(Z) = ZJ^T.
\end{equation}
It is not difficult to see that the convex clustering model
\eqref{eq: convex-clustering-edge} is a special case of \eqref{eq: two-block-structured}.
%by defining $f: \mathbb{R}^{p \times n} \to \mathbb{R}$ and $p: \mathbb{R}^{p \times |\mathcal{E}|} \to \mathbb{R}$ as
%\[
%f(X) : = \frac{1}{2}\|X - A\|_F^2, \quad p(\mathcal{B}(X)) := \sum_{(i, j) \in \mathcal{E}} w_{ij}\|X_{:i} - X_{:j}\|_p.
%\]
%Next, we will describe how to construct the reduced subproblems \eqref{eq: AS-small} for the convex clustering model, which is the key step to apply the AS technique.
\subsection{A Construction of the Reduced Problem}
 The main step for constructing the reduced subproblem is to construct the index sets $\alpha, \beta, \gamma$ and the corresponding matrix $M_{\gamma\alpha}$. For a given index set
\begin{equation}
\label{eq: index_I_cc}
I := \{l(i, j)\} \subseteq \{1, 2, \dots, |\mathcal{E}|\},
\end{equation}
 we can construct a subgraph $\hat{\mathcal{G}} \subseteq \mathcal{G}$ with edges $\hat{\mathcal{E}}:= \{(i, j) \mid l(i, j) \in I\}$ and all the corresponding nodes. Then, we can decompose the graph $\hat{\mathcal{G}}$ as
$$\hat{\mathcal{G}} = \hat{\mathcal{G}}_1 \cup \hat{\mathcal{G}}_2 \cup \cdots \cup \hat{\mathcal{G}}_s,$$
where $\hat{\mathcal{G}}_i$ are disjoint connected subgraph of $\hat{\mathcal{G}}$. Denote the node index set of $\hat{\mathcal{G}}_i$ as $\hat{\mathcal{N}}_i$ and we define
\[\alpha_i = \min \{k \mid k \in \hat{\mathcal{N}}_i\}, \quad i = 1, 2, \dots, s.\]
Then, we can uniquely determine the index sets $\alpha$, $\beta$ and $\gamma$ as
\[
\alpha = \{\alpha_1, \dots, \alpha_s\}, \quad \beta = [N] \backslash (\hat{\mathcal{N}}_1 \cup \cdots \cup \hat{\mathcal{N}}_s), \quad \mbox{and} \quad \gamma = (\hat{\mathcal{N}}_1 \cup \cdots \cup \hat{\mathcal{N}}_s) \backslash \alpha.
\]
The index sets $\alpha$, $\beta$ and $\gamma$ have clear meanings in the convex clustering model. For a given index set $I \subseteq [|\mathcal{E}|]$ and the generated graph $\hat{\mathcal{G}}$, $\alpha_i$ is the index of the selected representative point for the $i$-th cluster identified by the connected component $\hat{\mathcal{G}}_i$. On the other hand, $\beta$ is the collection of the indices of the isolated clusters which contain only a singleton.

Furthermore, we could have an explicit formula for $M_{\gamma\alpha} \in \mathbb{R}^{|\alpha|\times|\gamma|}$, which is given by
\[
(M_{\gamma\alpha})_{ij} = \left\{
\begin{array}{ll}
1, & \mbox{if $j \in \hat{\mathcal{N}}_i$}, \\
0, & \mbox{otherwise}.
\end{array}
\right.
\]
Then
\[
X_{:\gamma} = X_{:\alpha}M_{\gamma\alpha},
\]
which actually maps the data points indexed by $\gamma$ to the corresponding centroids with indices in the set $\alpha$.

\section{Numerical Experiments}
In this section, we demonstrate the efficiency of the proposed solver independent AS technique and EAS technique via the important convex clustering model \eqref{eq: convex-clustering-edge} (with $p = 2$). In this paper, we mainly focus on the numerical efficiency of our proposed techniques, readers can refer to \cite{sun2021convex,hocking2011clusterpath,lindsten2011clustering} and the references therein for the performance of clustering by the convex clustering model \eqref{eq: convex-clustering-edge}. We
test the AS technique with AMA \cite{chi2015splitting}, ADMM \cite{chi2015splitting} and {\sc Ssnal} \cite{yuan2018efficient}, which are the three of the most popular algorithms for solving \eqref{eq: convex-clustering-edge}. Due to the limited length of the paper, we omit the details of these three algorithms but refer the readers to consult the aforementioned references. In our experiments, by default, we will generate the clustering path with $\lambda =: [10:-0.2:1]$. The weights $w_{ij}$ will be defined by the following Gaussian kernel with $k$-nearest neighbors (we choose $k=10$ in our experiments):
\begin{equation}\label{eq: weights_define}
  w_{ij} = \left\{
  \begin{array}{ll}
  \exp(-\frac{1}{2}\|A_{:i} - A_{:j}\|^2), & \mbox{if $(i, j) \in \mathcal{E}$},\\
  0, & \mbox{if $(i, j) \not \in \mathcal{E}$},
  \end{array}
  \right.
\end{equation}
where $\mathcal{E} = \{(i, j) \mid \mbox{$A_{:i}$ is among $A_{:j}$'s $k$ nearest neighbors}\}$.

For a fair comparison with the fast AMA algorithm, in this paper, we terminate all the algorithms based on the relative duality gap:
\begin{equation}
\label{eq: stop-criteria}
\begin{array}{c}
\eta = \frac{F_{\lambda}(X) - D_{\lambda}(Z)}{1 + |F_{\lambda}(X)| + |D_{\lambda}(Z)|} < \epsilon.
\end{array}
\end{equation}
Here, $\epsilon > 0$ is a given tolerance, $F_{\lambda}(X)$ and $D_{\lambda}(Z)$ are the objective function value of the primal problem \eqref{eq: two-block-structured} and the dual problem \eqref{eq: dual-problem}, respectively. We set $\epsilon = 10^{-6}$ in \eqref{eq: stop-criteria} and $\hat{\epsilon} = 2{\rm e}$-$16$ in \eqref{eq: initial_I_i} by default in this paper. All our computational results are obtained by running
{\sc Matlab} on a windows workstation (Intel Xeon E5-2680 @ 2.50GHz).
\subsection{Simulated Data Sets}
In this subsection, we provide some numerical results on the simulated two half-moon data, which is one of the most popular data sets for clustering.\\[3pt]
\begin{figure}[!ht]
     \centering
     \begin{subfigure}[b]{0.4\textwidth}
         \centering
         \includegraphics[scale = 0.4]{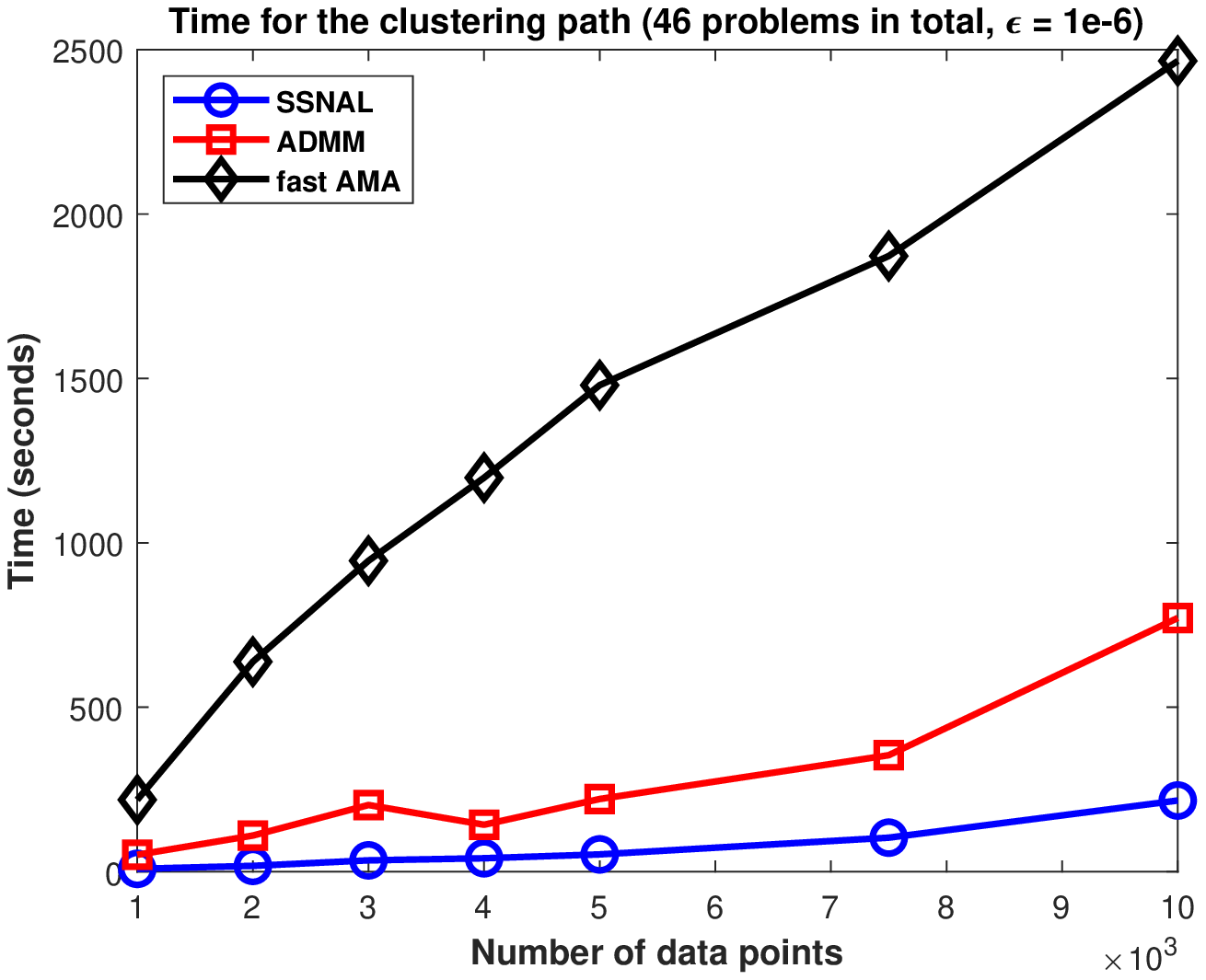}
         \caption{}
         \label{fig:direct_solve}
     \end{subfigure}
     \begin{subfigure}[b]{0.4\textwidth}
         \centering
         \includegraphics[scale = 0.4]{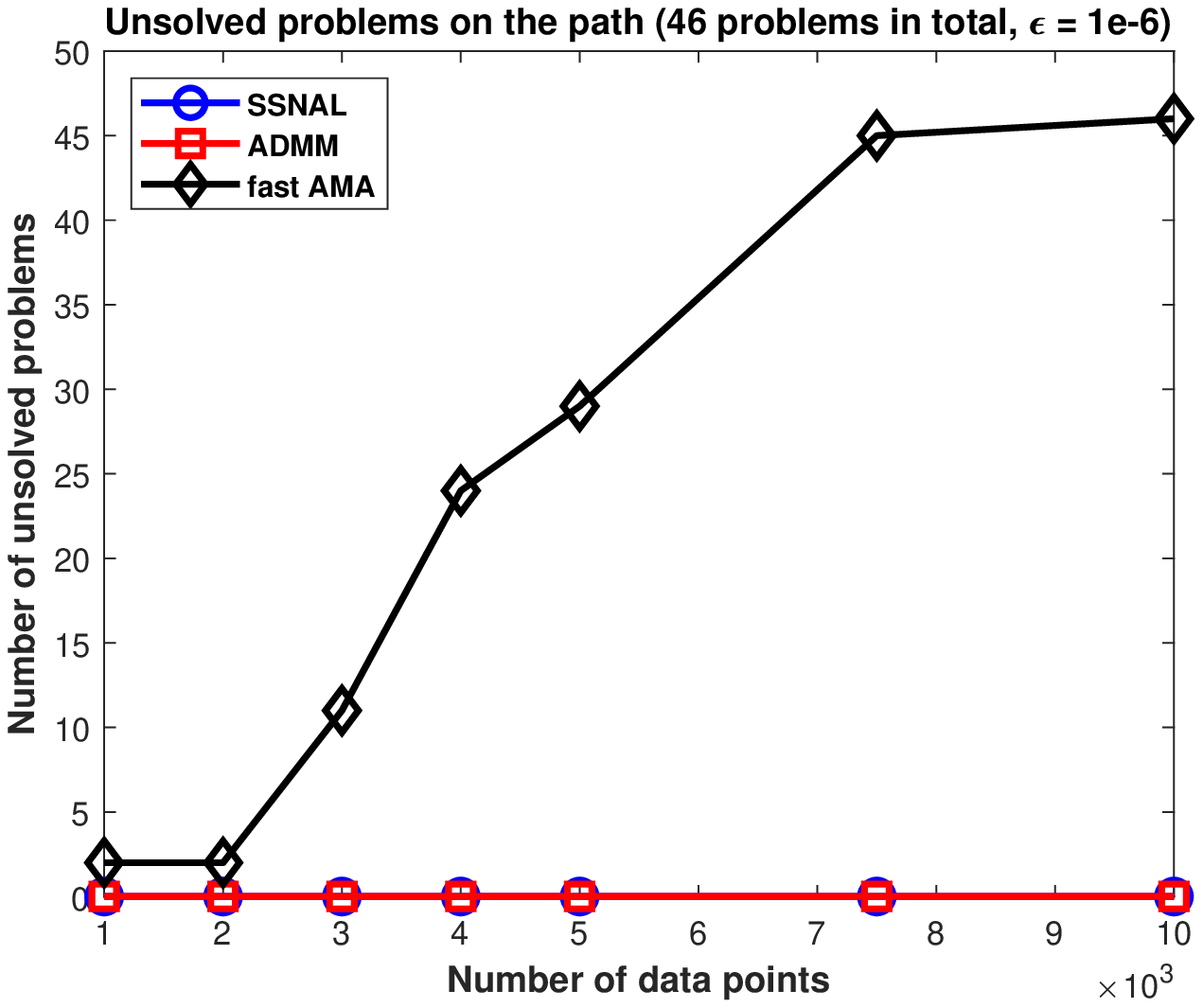}
         \caption{}
         \label{fig:unsolved_direct}
     \end{subfigure}
     \begin{subfigure}[b]{0.4\textwidth}
         \centering
         \includegraphics[scale = 0.4]{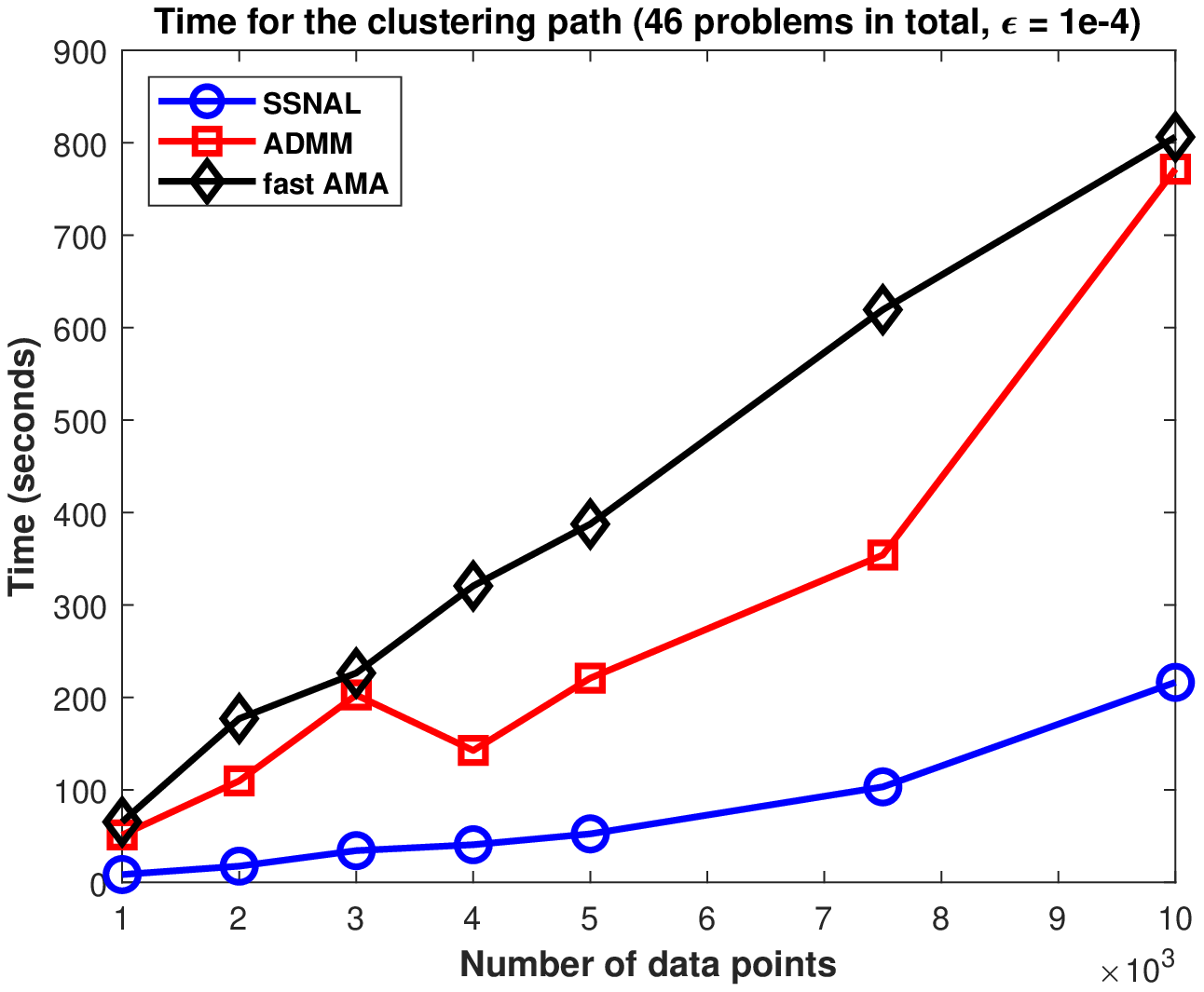}
         \caption{}
         \label{fig:admmvsamalowacc}
     \end{subfigure}
      \begin{subfigure}[b]{0.4\textwidth}
         \centering
         \includegraphics[scale = 0.4]{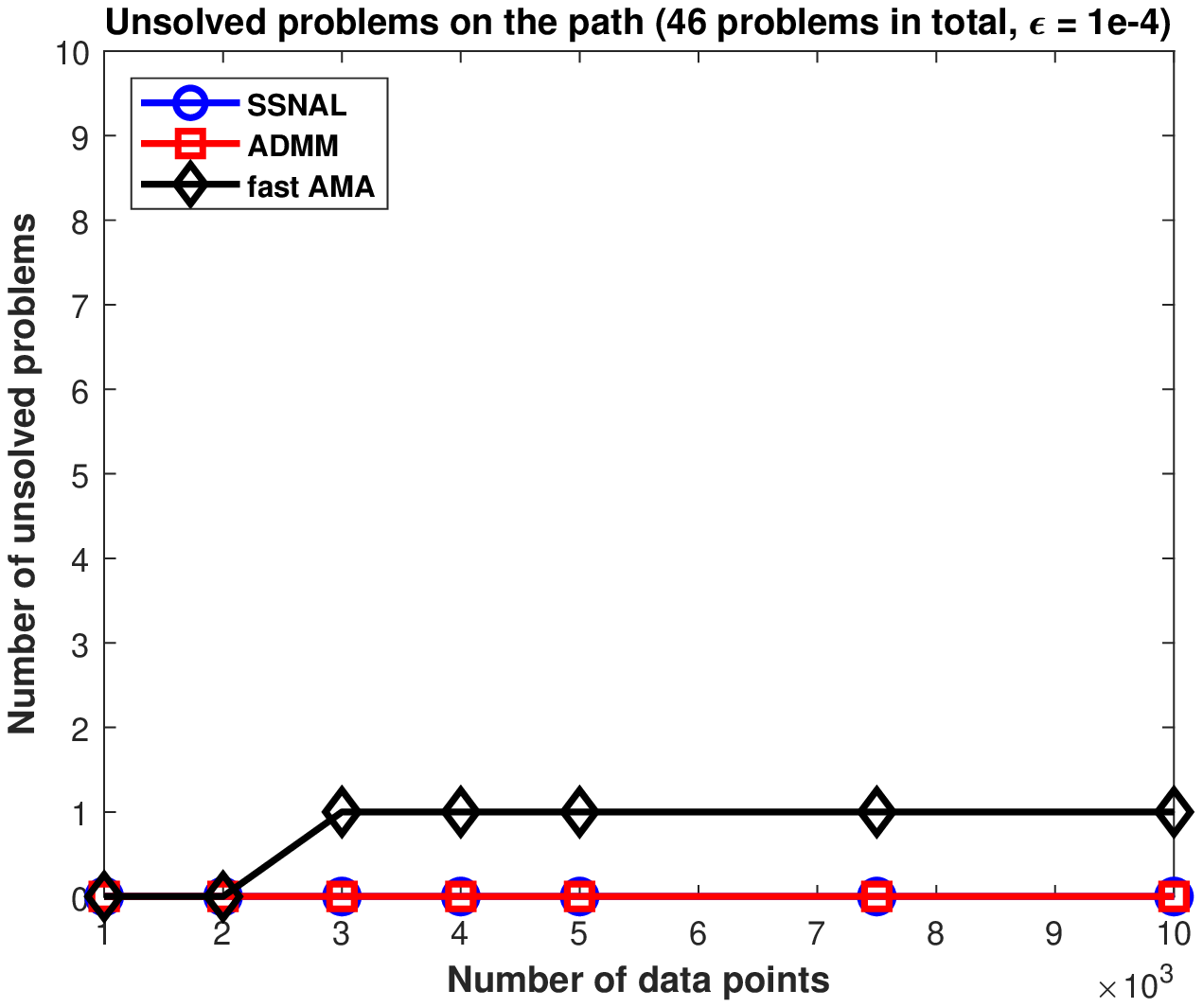}
         \caption{}
         \label{fig:unsolved_lowacc}
     \end{subfigure}
    \caption{Numerical performance revisit on the two half moon data set.}
        \label{fig:two-half-moon-revisit}
\end{figure}
First, we revisit the performance of fast AMA \cite{chi2015splitting}, ADMM \cite{chi2015splitting} and {\sc Ssnal} \cite{yuan2018efficient} for generating the clustering path directly. We implemented the three algorithms in {\sc Matlab} and tried our best to optimize the computations for a fair comparison\footnote{Readers can find the implementations at: https://blog.nus.edu.sg/mattohkc/softwares/convexclustering/}.  As shown in Figure \ref{fig:direct_solve}, {\sc Ssnal} is the best among the three algorithms on this data set. However, unlike the statements in \cite{chi2015splitting} stating that fast AMA is much better than ADMM, we actually observe some discrepancies in the performance. Fast AMA could not achieve the accuracy we set for most of the cases when $n$ is relatively large (Figure \ref{fig:unsolved_direct}). For a fairer comparison, we revisit the performance of the three algorithms under the relatively low accuracy setting with $\epsilon = 10^{-4}$ (Figure \ref{fig:admmvsamalowacc}, \ref{fig:unsolved_lowacc}), our numerical results show that ADMM is still better than fast AMA even in the low accuracy setting. Since the fast AMA has difficulty solving \eqref{eq: convex-clustering-edge} to  high accuracy, we focus on applying the AS technique with ADMM and {\sc Ssnal}.\\[3pt]
\begin{figure}[!ht]
     \centering
      \begin{subfigure}[b]{0.32\textwidth}
         \centering
         \includegraphics[scale = 0.32]{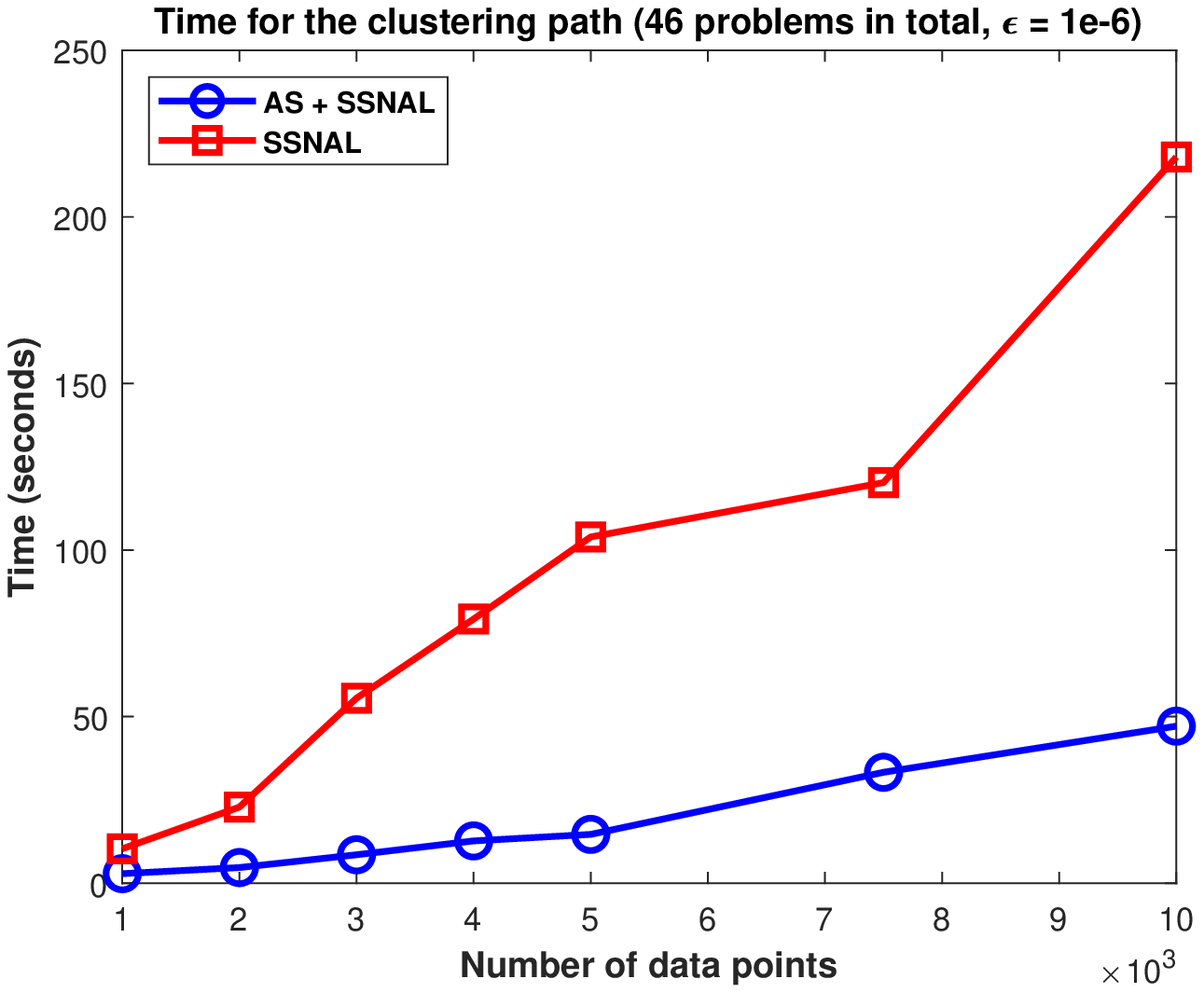}
         \caption{}
         \label{fig:assnal}
     \end{subfigure}
     \begin{subfigure}[b]{0.32\textwidth}
         \centering
         \includegraphics[scale = 0.32]{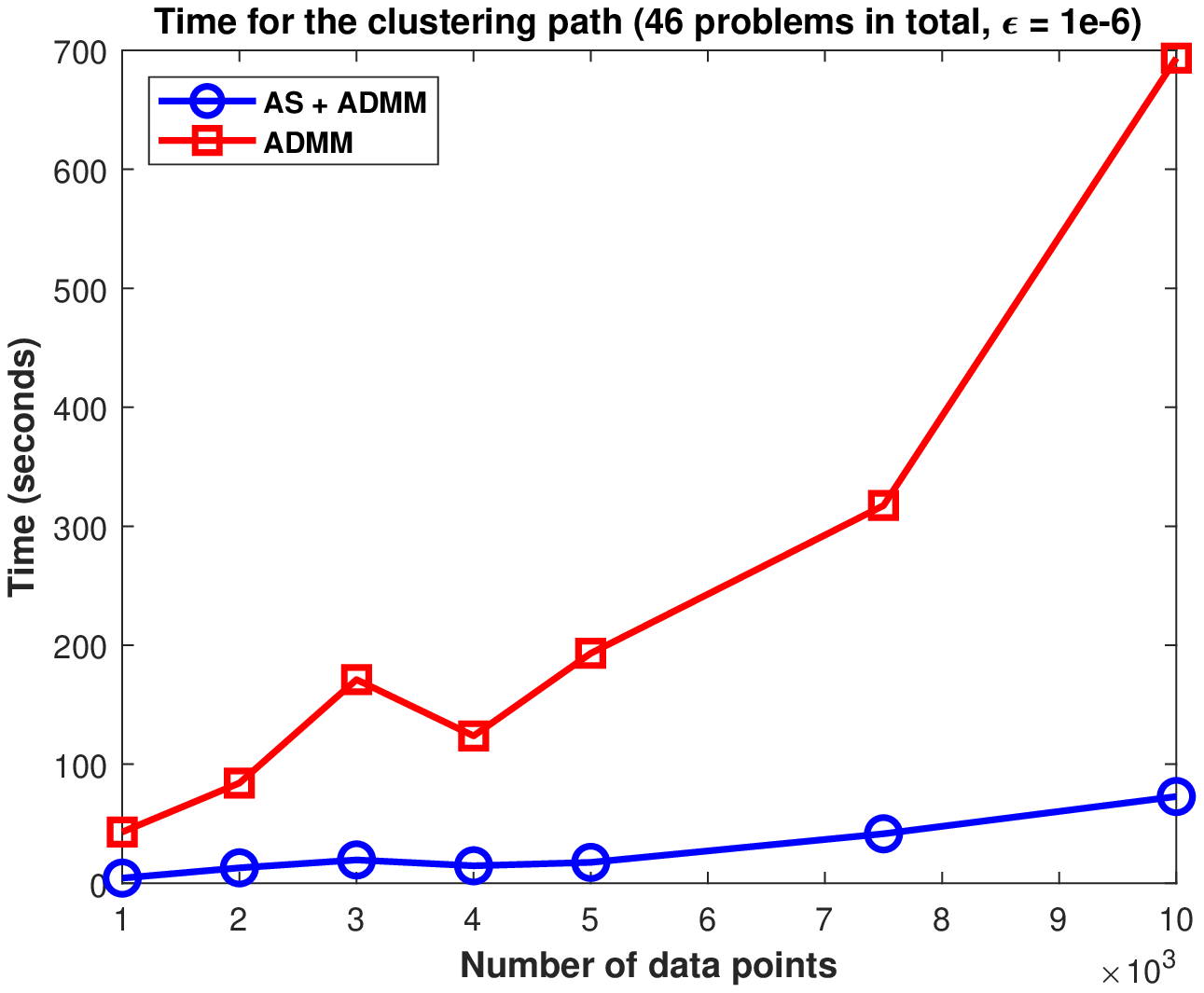}
         \caption{}
         \label{fig:asadmm}
     \end{subfigure}
      \begin{subfigure}[b]{0.32\textwidth}
         \centering
         \includegraphics[scale = 0.32]{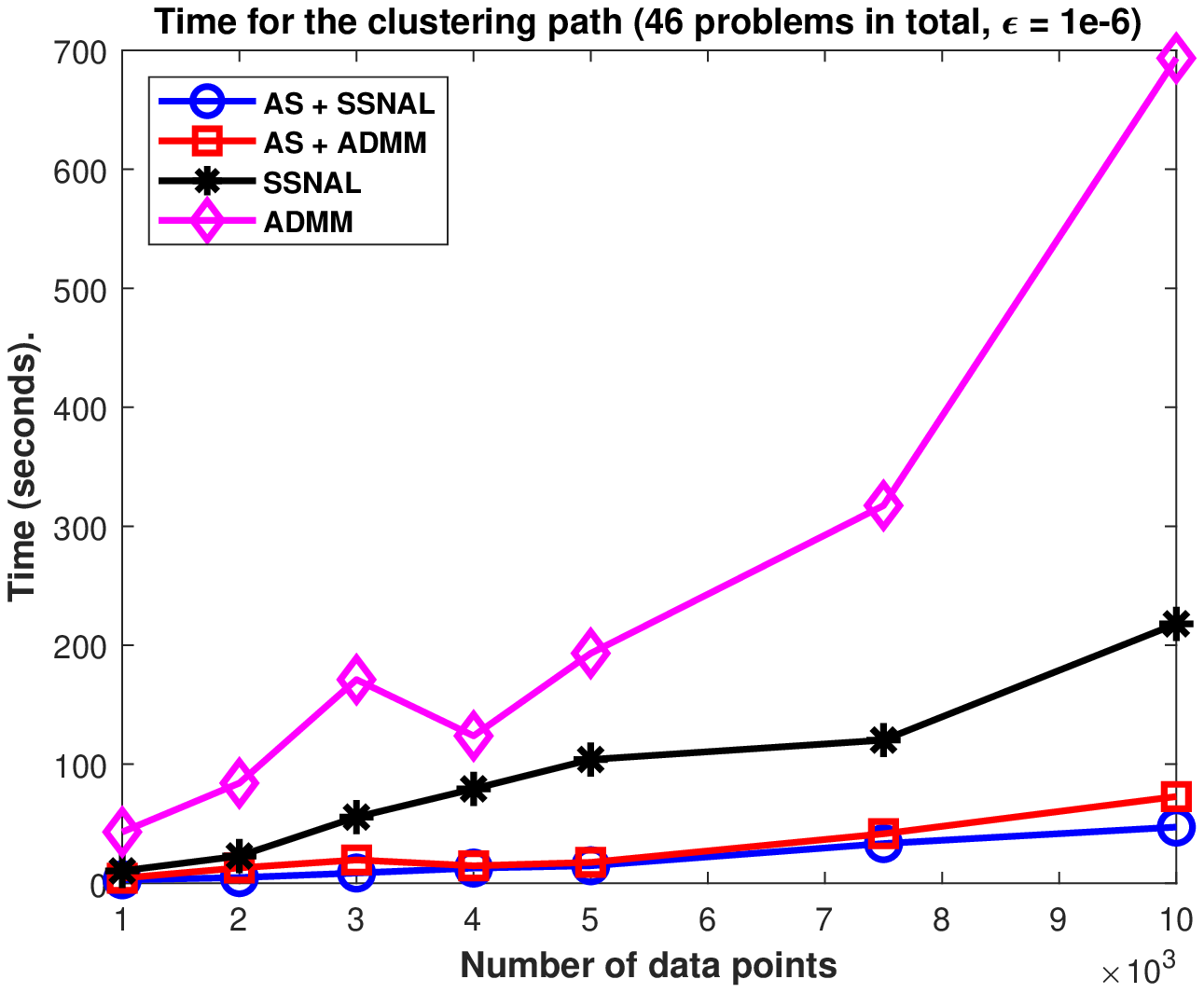}
         \caption{}
         \label{fig:asall}
     \end{subfigure}
        \caption{Numerical performance on the two half moon data set with $k=10$.}
        \label{fig:two-half-moon}
\end{figure}

Now, we move on to present the numerical performance of the proposed AS technique. The details could be found in Figure \ref{fig:two-half-moon}. Our numerical results on the two
half-moon data set show that the AS technique could accelerate the {\sc Ssnal} and the ADMM by up to $\mathbf{4.8}$ times (Figure \ref{fig:assnal}) and $\mathbf{12.8}$ times (Figure \ref{fig:asadmm}), respectively. With the help of the AS technique, AS+ADMM could even be  comparable to AS+{\sc Ssnal} (Figure \ref{fig:asall}), which demonstrates the power of the AS technique for capturing the intrinsic structured sparsity of the convex clustering model. Since the AS technique can take   advantage of the sparse structure to  substantially reduce the dimension of the problem,  we can apply the sparse Cholesky decomposition to solve the linear system involved in ADMM in a highly efficient way. This also partially demonstrates that ADMM is efficient to solve small scale convex clustering problems.

\begin{figure}[!ht]
     \centering
     \hfill
      \begin{subfigure}[b]{0.32\textwidth}
         \centering
         \includegraphics[scale = 0.32]{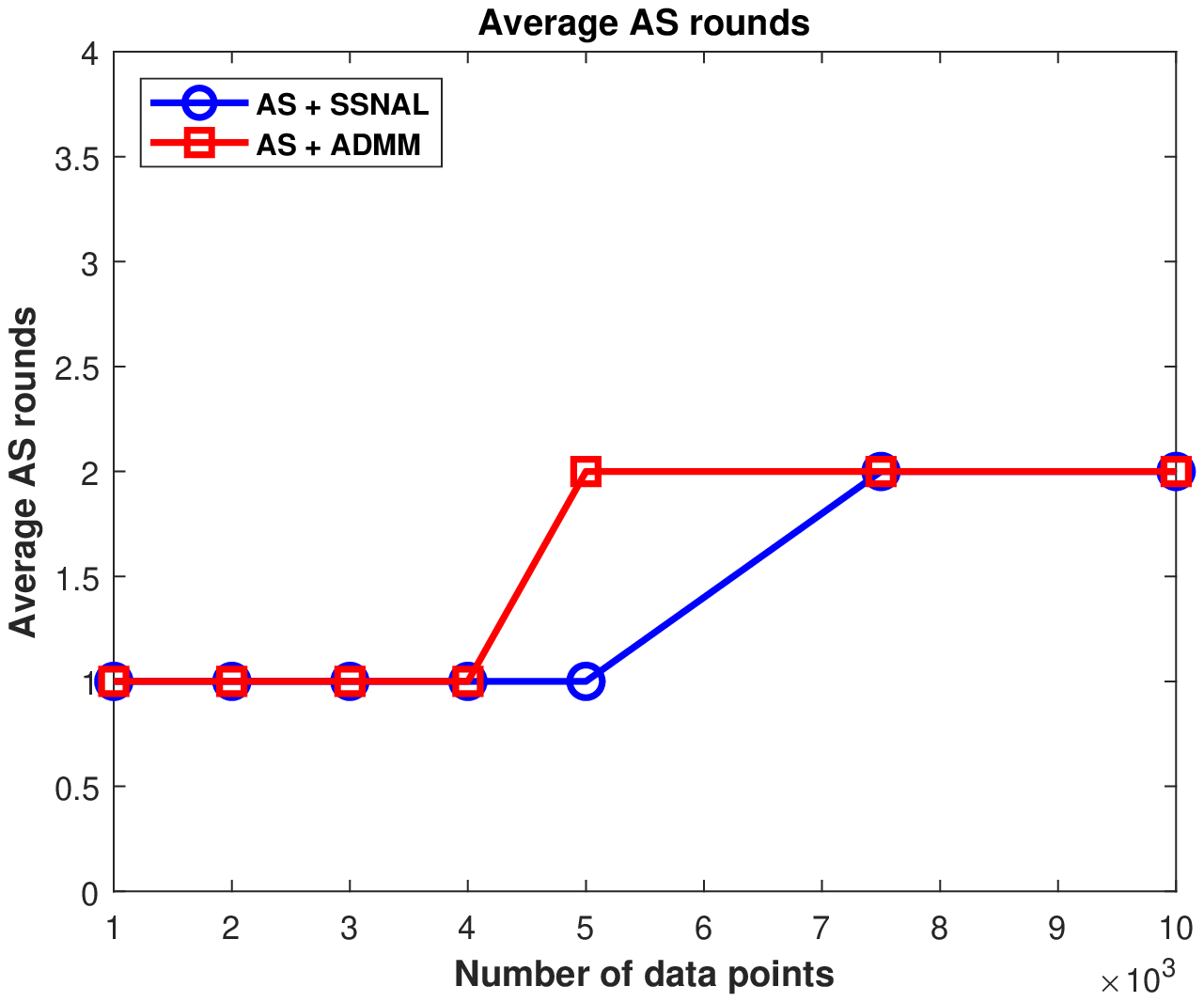}
         \caption{}
         \label{fig: AS_round}
     \end{subfigure}
     \hfill
     \begin{subfigure}[b]{0.32\textwidth}
         \centering
         \includegraphics[scale = 0.32]{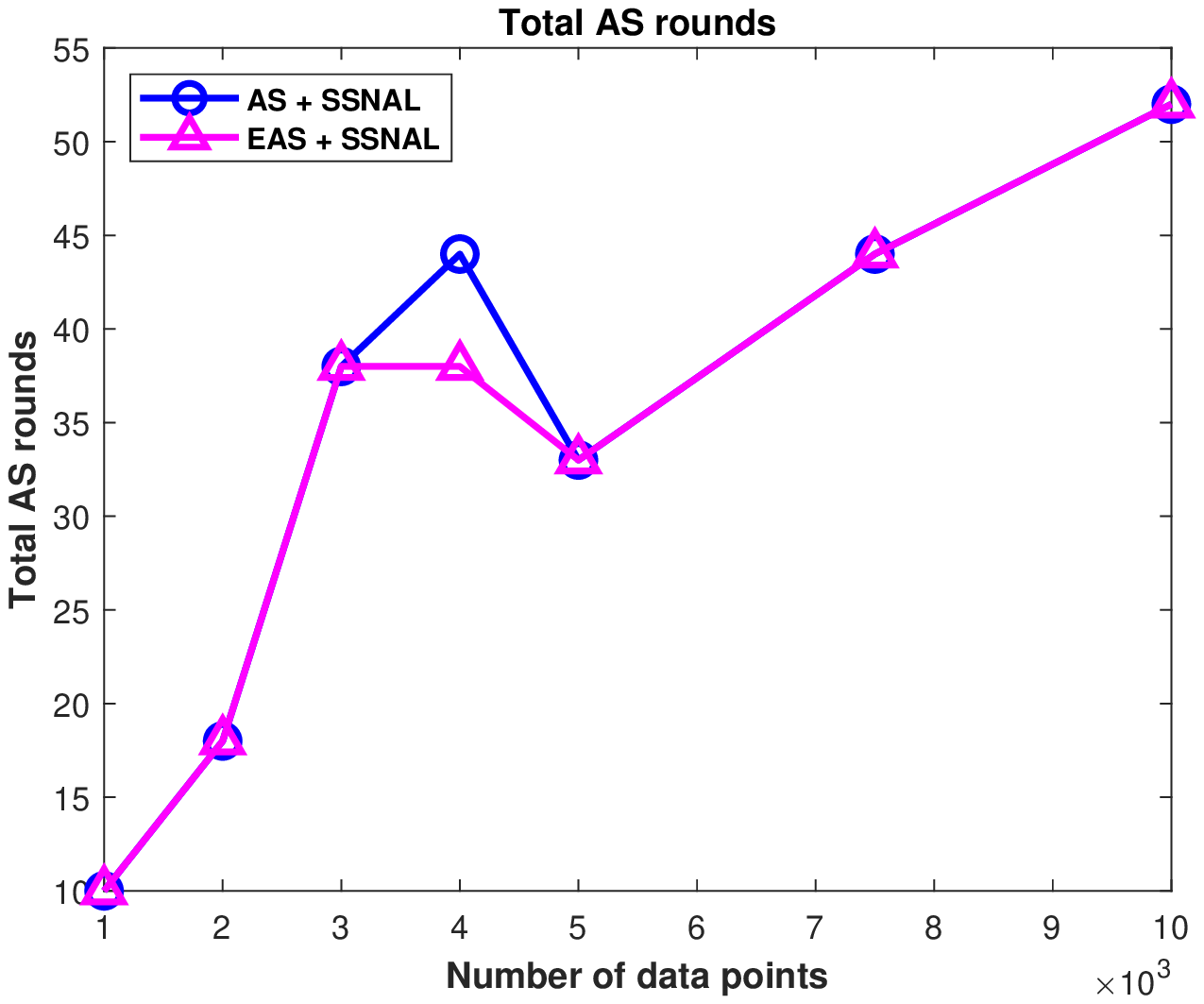}
         \caption{}
         \label{fig: hybrid_rounds_ssnal}
     \end{subfigure}
     \hfill
      \begin{subfigure}[b]{0.32\textwidth}
         \centering
         \includegraphics[scale = 0.32]{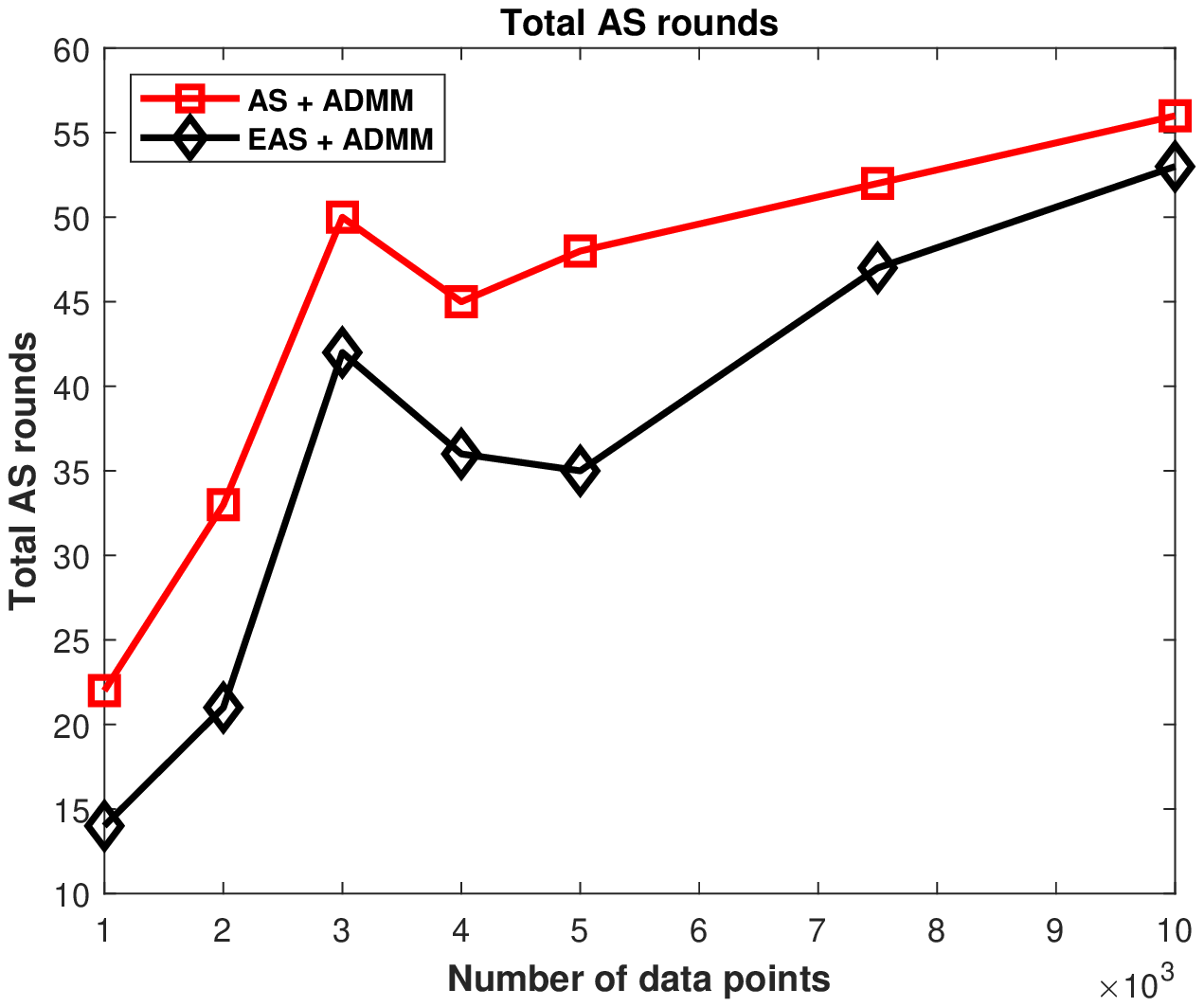}
         \caption{}
         \label{fig: hybrid_rounds_admm}
     \end{subfigure}
     \medskip
      \hfill
      \begin{subfigure}[b]{0.32\textwidth}
         \centering
         \includegraphics[scale = 0.32]{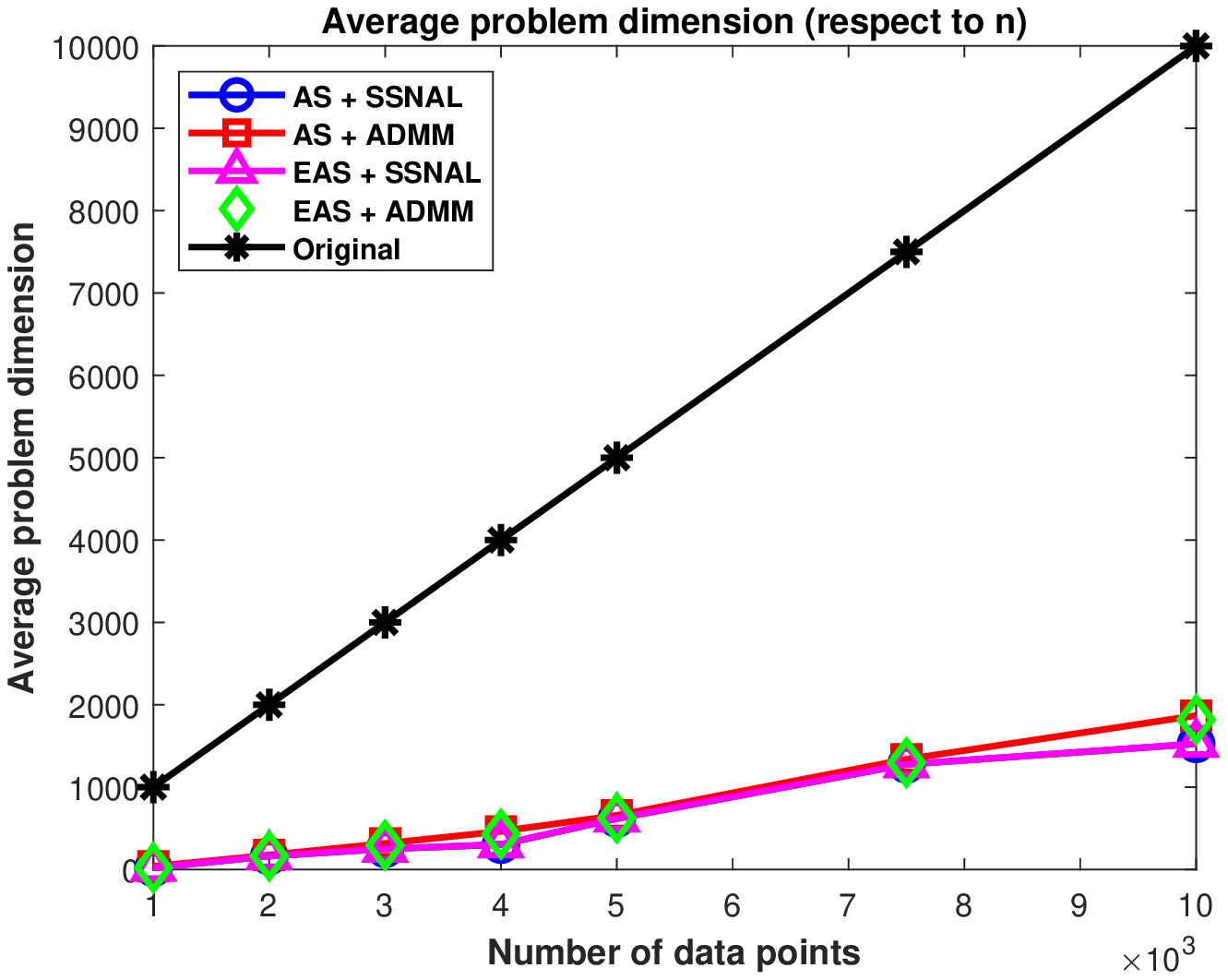}
         \caption{}
         \label{fig: AS_dimension}
     \end{subfigure}
     \hfill
      \begin{subfigure}[b]{0.32\textwidth}
         \centering
         \includegraphics[scale = 0.32]{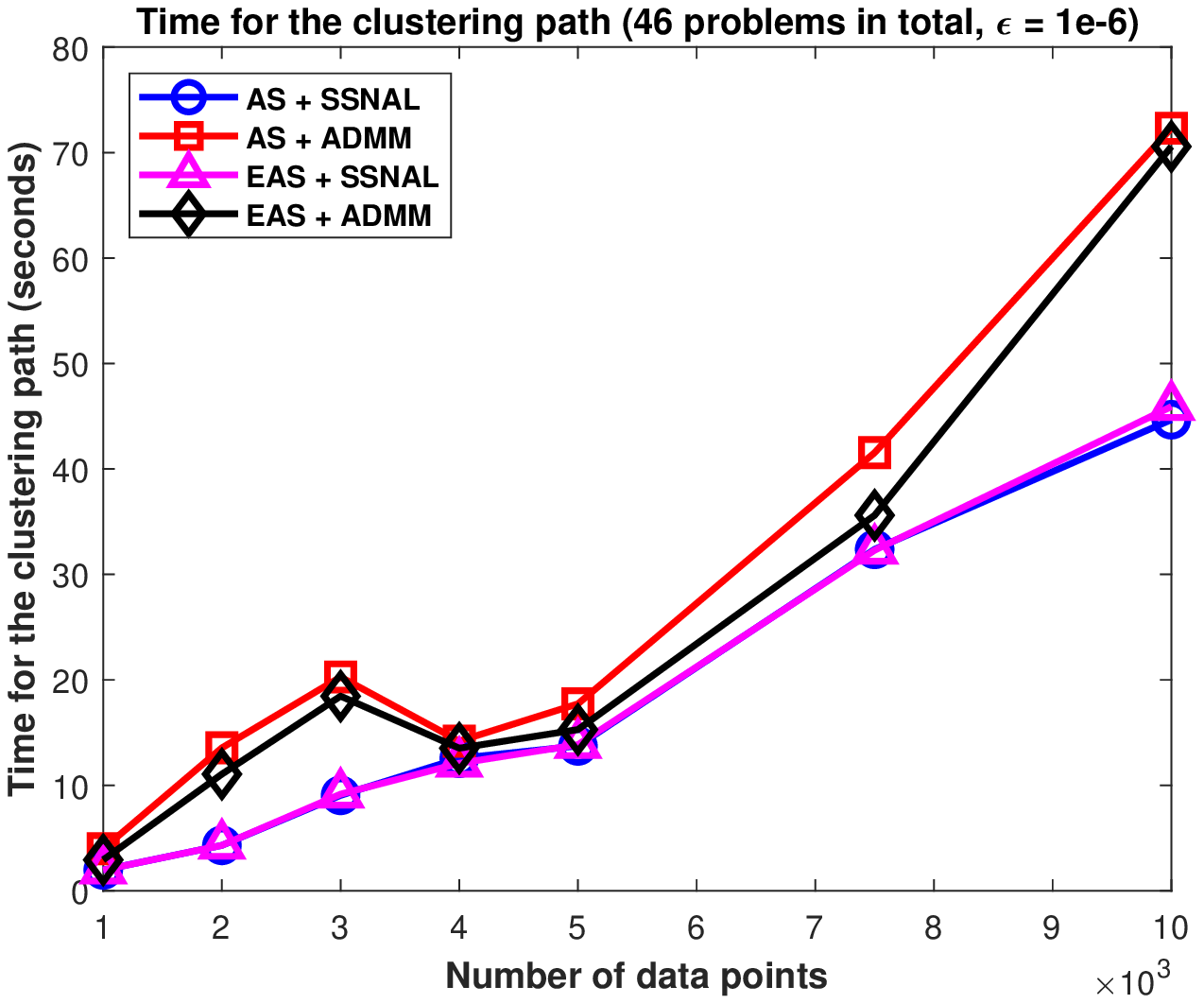}
         \caption{}
         \label{fig: AS_vs_hybrid}
     \end{subfigure}
     \hfill \mbox{}
        \caption{Numerical performance on the two half moon data set with $k=10$.}
        \label{fig:two-half-moon-part2}
\end{figure}

Next, we move on to present the empirical comparison between the AS technique and the EAS technique on the two half-moon data set. The results could be found in Figure \ref{fig:two-half-moon-part2}. As shown in Figure \ref{fig: AS_round}, the AS technique performs very well and the average AS rounds are very small, even for large scale problems. Furthermore, as shown in Figure \ref{fig: AS_dimension}, the sizes of the reduced problems are much smaller than those  of the original problems. These are the main reasons why the AS technique can accelerate the algorithms. On the other hand, as one may imagine, the EAS could potentially early-stop the AS procedure. Thus the EAS technique could reduce the AS rounds and further accelerate the algorithms. This phenomenon is indeed observed in numerical experiments. As shown in Figure \ref{fig: hybrid_rounds_ssnal} and Figure \ref{fig: hybrid_rounds_admm}, the EAS can have fewer AS rounds. The running time comparison could be found in Figure \ref{fig: AS_vs_hybrid}, which is consistent with our expectation.

\begin{remark}
  We close this subsection by making some remarks.
  \begin{itemize}
      \item[1.] One may be curious about the phenomenon where AS+{\sc Ssnal} could have fewer AS rounds than AS+ADMM (Figure \ref{fig: AS_round}). Now we try to give a plausible explanation. Although we set the same tolerance for terminating {\sc Ssnal} and ADMM, the real accuracy achieved by the two algorithms are different. In our experiments, we observe that the {\sc Ssnal} achieves higher accuracy than ADMM due to its faster convergence rate. This may be the main reason for the phenomenon shown in Figure \ref{fig: AS_round}.
      \item[2.] As shown in Figure \ref{fig: AS_vs_hybrid}, EAS could further accelerate ADMM but may not be so for
      {\sc Ssnal}, although it may early-terminate the AS procedure. This mainly because we need to solve additional auxiliary optimization problems in Algorithm \ref{alg:enhancedAS} by the APG algorithm and it may spend more time than solving a few more reduced problems with {\sc Ssnal}, because {\sc Ssnal} is very efficient on this data set.
      \item[3.] One may naturally agree that the AS technique and the EAS technique could be very powerful when $\lambda$ is relatively large, since many data points are assigned to only a few clusters in this case. However, since we generate the whole clustering path, some problems on the clustering path may not have this nice property when the parameter $\lambda$ is small (which affects the efficiency of the AS technique and the EAS technique). But we still observe the distinctive advantages of them.
  \end{itemize}
\end{remark}

\subsection{Real Data Sets}
In this subsection, we will present the performance of the AS technique and the EAS technique for generating the clustering path on the MNIST dataset \cite{lecun1998mnist}. We adopt the preprocessing method described in \cite{mixon2017clustering}, which applies a one hidden layer linear neural network to preprocess the raw images. Then, we apply the convex clustering model \eqref{eq: convex-clustering-edge} on the preprocessed data. Our experiments is on the testing set of MNIST data and the dimension of the preprocessed data is $10 \times 10000$. The details could be found in Table \ref{tab: mnist}.
\begin{table}[!ht]
\centering
\begin{tabular}{ |c|c|c| }
\hline
 &  {\sc Ssnal} & ADMM \\
\hline
& direct $|$ with AS $|$ with EAS & direct $|$ with AS $|$ with EAS\\
\hline
Time (seconds) & $1207.7$ $|$ $156.3$ $|$ $157.3$ & $1823.8$ $|$ $128.5$ $|$ $132.2$\\
\hline
Total AS round & $0$ $|$ $45$ $|$ $45$ & $0$ $|$ $47$ $|$ $47$\\
\hline
Average problem dimension & $10000$ $|$ $1377$ $|$ $1377$ & $10000$ $|$ $1389$ $|$ $1389$\\
\hline
\end{tabular}
\caption{Numerical performance on the MNIST data set.}
\label{tab: mnist}
\end{table}
From the results, we observe that the AS technique could accelerate the ADMM by up to $\mathbf{14.2}$ times and the {\sc Ssnal} by up to $\mathbf{7.7}$ times. It is understandable that AS could be more attractive for ADMM, since the second-order sparsity embedded in the algorithm {\sc Ssnal} has partially captured the structured sparsity already. Moreover, since the EAS technique does not reduce the sieving iterations on this data set comparing to the AS technique, the EAS technique will spend more time than the AS technique.

\section{Conclusion}
In this paper, we propose an AS technique and an enhanced AS technique, which are solver independent, for convex optimization problems with structured sparsity.
The proposed techniques can accelerate various optimization algorithms by substantially reducing the dimension of the problems that need to be solved. Numerical performance on the convex clustering model has demonstrated the high efficiency of the proposed dimension reduction techniques. We also established a finite convergence property of the AS and enhanced AS techniques in this paper. However, we should note that in the worst-case, the AS technique may sieve all the indices. But based on our empirical evaluation, one can say that the AS technique works very well in practice. Thus, as a future research topic, we will make efforts to analyze the average-case complexity of the AS and the enhanced AS technique.
%\bibliography{reference_ASSP}
%\bibliographystyle{siam}

\end{document}